\documentclass[12pt,a4paper,twoside]{amsart}
\pdfoutput=1
\usepackage{amssymb}
\usepackage{esint}
\usepackage{mathtools}
\usepackage{hyperref}
\usepackage[english]{babel}
\usepackage[latin1]{inputenc}
\usepackage[pdftex]{graphicx}
\theoremstyle{plain}
\newtheorem{theorem}{Theorem}
\newtheorem{lemma}[theorem]{Lemma}

\newtheorem{proposition}[theorem]{Proposition}

\theoremstyle{definition}
\newtheorem{definition}[theorem]{Definition}

\theoremstyle{remark}
\newtheorem{remark}[theorem]{Remark}
\newcommand{\R}{\mathbb R}
\newcommand{\C}{\mathbb C}
\newcommand{\N}{\mathbb N}
\newcommand{\Z}{\mathbb Z}
\newcommand{\Q}{\mathbb Q}
\newcommand{\h}{\mathcal H}
\newcommand{\A}{\mathcal A}
\newcommand{\brt}{{\mathcal G}}
\newcommand{\rt}{{\mathcal I}}
\newcommand{\der}{\mathrm{d}}
\newcommand{\eps}{\varepsilon}
\renewcommand{\phi}{\varphi}
\newcommand{\abs}[1]{\left\lvert #1 \right\rvert}
\newcommand{\aabs}[1]{\left\| #1 \right\|}
\newcommand{\Order}{\mathcal O}
\newcommand{\order}{o}
\newcommand{\id}{\operatorname{id}}
\newcommand{\ip}[2]{\left\langle#1,#2\right\rangle}

\newcommand{\Der}[1]{\frac{\der}{\der #1}}
\newcommand{\roo}[1]{{\tilde{#1}}}
\newcommand{\Lip}{\operatorname{Lip}}
\newcommand{\lip}{\operatorname{lip}}
\newcommand{\lin}{\mathcal L}
\newcommand{\prt}{\mathcal P}

\title[Abel and X-ray transforms in low regularity]{Abel transforms with low regularity with applications to X-ray tomography on spherically symmetric manifolds}
\author{Maarten V. de Hoop}
\thanks{Department of Computational and Applied Mathematics, Rice University}
\email{mdehoop@rice.edu}
\author{Joonas Ilmavirta}
\thanks{Department of Mathematics and Statistics, University of Jyv\"askyl\"a}
\email{joonas.ilmavirta@jyu.fi}
\date{\today}

\begin{document}

\begin{abstract}
We study ray transforms on spherically symmetric manifolds with a piecewise $C^{1,1}$ metric.
Assuming the Herglotz condition, the X-ray transform is injective on the space of $L^2$ functions on such manifolds.
We also prove injectivity results for broken ray transforms (with and without periodicity) on such manifolds with a $C^{1,1}$ metric.
To make these problems tractable in low regularity, we introduce and study a class of generalized Abel transforms and study their properties.
This low regularity setting is relevant for geophysical applications.
\end{abstract}

\maketitle

\tableofcontents

\section{Introduction}
\label{sec:intro}


Our aim is to study geodesics and ray transforms on a spherically symmetric manifold.
More specifically, our manifold is the Euclidean annulus $M=\bar B(0,1)\setminus\bar B(0,R)\subset\R^n$, $R\in(0,1)$ and $n\geq2$, with the metric $g(x)=c^{-2}(\abs{x})e(x)$, where~$e$ is the standard Euclidean metric and $c\colon(R,1]\to(0,\infty)$ is a piecewise~$C^{1,1}$ function satisfying the Herglotz condition (see definition~\ref{def:herglotz}).
A piecewise~$C^{1,1}$ function may have a finite number of jump discontinuities.

By a geodesic we mean a maximal unit speed geodesic on the Riemannian manifold~$(M,g)$ with endpoints at the outer boundary $\partial M\coloneqq\partial B(0,1)$.
One of the problems we study is to recover a function $f\in L^2(M)$ from its integrals over all geodesics.
The same problem with $f\in C^\infty$ and $c\in C^\infty$ was considered by Sharafutdinov~\cite{S:rce-tensor}.
For $f\in L^2$ and $c\in C^\infty$ the result follows from the local support theorem by Uhlmann and Vasy~\cite{UV:local-x-ray} since the Herglotz condition is equivalent with the foliation condition in spherical symmetry.
Reducing regularity introduces technical difficulties but makes the problem more relevant for seismic imaging.

In addition to the X-ray transform, we consider the broken ray transform and the periodic broken ray transform.
In order to prove injectivity results for these transforms, we need various integral transforms: Abel-type transforms, the Fourier series, the Funk transform, and the planar average ray transform.

See section~\ref{sec:results} for an overview of our methods and results.

\subsection{The Herglotz condition}

We now define precisely what we mean by the Herglotz condition (named after Herglotz~\cite{H:kinematic}) and geodesics in the low regularity setting.
We will also introduce so-called countable and finite conjugacy conditions.

\begin{definition}
\label{def:herglotz}
By a piecewise~$C^{1,1}$ function satisfying the \emph{Herglotz condition} we mean a function $c\colon(R,1]\to(0,\infty)$ satisfying the following:
\begin{itemize}
\item
The interval~$(R,1]$ is a finite disjoint union of intervals~$(a,b]$ so that~$c$ is~$C^{1,1}$ regular on each such interval.
\item
On each subinterval~$(a,b]$ the function satisfies the Herglotz condition
\begin{equation}
\label{eq:herglotz}
\Der{r}\left(\frac{r}{c(r)}\right)>0
.
\end{equation}
\item
At endpoints~$r$ of the subintervals (but not of the whole interval)
\begin{equation}
\label{eq:herglotz-jump}
\limsup_{s\to r+}c(s)
\leq
c(r)
.
\end{equation}
(The function~$c$ is upper semicontinuous from the right.)
\end{itemize}
\end{definition}

The jump condition~\eqref{eq:herglotz-jump} simply requires the inequality~\eqref{eq:herglotz} when the left-hand side is interpreted as a signed measure or a distribution rather than a function.
The Herglotz condition for regular~$c$ means that all geodesics reach the boundary, and with less regularity we have to exclude trapping due to total internal reflection.
In other words, the Herglotz condition is equivalent with the manifold being non-trapping.

Consider a subinterval $(a,b]\subset(R,1]$ with $a\neq R$.
By~\eqref{eq:herglotz} the function~$r/c(r)$ increases on this interval, so it has a limit at~$a$ in~$[0,\infty)$.
Combining this with the jump condition~\eqref{eq:herglotz-jump} at $r=a$ shows that $\lim_{r\to a+}c(r)$ exists in~$(0,\infty)$.
Using the conserved quantities introduced in section~\ref{sec:geodesic} below, we conclude that any geodesic approaching the surface $r=a$ from either direction has a well-defined limit --- for both position and direction.

The above argument actually shows that~$c$ is bounded away from zero and infinity on $(R+\eps,1]$ for any $\eps>0$.
The definition does permit $\lim_{r\to R}c(r)=\infty$.

Total internal reflection from the outside is possible, but we exclude all geodesics that reflect where the wave speed~$c$ jumps.
In practice such geodesics would give valuable data for the inverse problem, but that data are unnecessary for uniqueness and we therefore omit it.
We also exclude geodesics that are tangent to a surface where~$c$ jumps.
Geodesics that meet such a surface are therefore assumed to traverse it according to Snell's law.
Our geodesics do not branch; they are completely transmitted.

This is not merely a matter of technical convenience.
If one wants to take branching into account, the whole concept of X-ray transform needs to be redefined.
This generalization is particularly inobvious for the periodic broken ray transform.

We do not know whether the Herglotz condition and piecewise~$C^{1,1}$ regularity are necessary conditions for our injectivity results.
Without these assumptions the geometric framework starts breaking apart and the problem may need a reformulation.
There are geodesic X-ray tomography results on various manifolds without symmetries (see section~\ref{sec:app}), but our method of proof relies heavily on symmetry.
In general, one may extend injectivity results to nearby geometries by a perturbation argument if one has stability in addition to injectivity.
However, we have little or not restrictions for conjugate points and cannot therefore expect stability.

In addition to the Herglotz condition, our results for the broken ray transforms (periodic and not) require an additional assumption.
This assumption may in fact follow from the Herglotz condition, but we state it separately since we have no proof of the implication.

\begin{definition}
\label{def:ccc}
We say that a~$C^{1,1}$ wave speed~$c$ satisfies the \emph{countable conjugacy condition} 
if there are only countably many maximal geodesics whose endpoints are conjugate, up to rotational symmetry.
\end{definition}

For some negative results (see theorem~\ref{thm:pbrt}) we need a stronger condition.

\begin{definition}
\label{def:fcc}
We say that a~$C^{1,1}$ wave speed~$c$ satisfies the \emph{finite conjugacy condition} 
if there are only finitely many maximal geodesics whose endpoints are conjugate, up to rotational symmetry.
\end{definition}

The countable conjugacy condition holds for most wave speeds, and assuming it is unlikely to be an issue in applied problems.

\subsection{Overview of results and methods}
\label{sec:results}

Our main results concern ray transforms on spherically symmetric manifolds.
See the theorems mentioned here for more details and related results.

\begin{itemize}
\item
Assuming a piecewise~$C^{1,1}$ wave speed satisfying the Herglotz condition, the attenuated geodesic X-ray transform is injective when the attenuation is radially symmetric and Lipschitz continuous.
(See theorem~\ref{thm:xrt}.)
\item
Assuming a~$C^{1,1}$ wave speed satisfying the Herglotz condition and the countable conjugacy condition, the broken ray transform is injective on functions that are quasianalytic in the angular variable(s).
The broken rays have endpoints on the boundary (they are not periodic), and the endpoints can be limited to any open set.
In two dimensions, for example, a function of the form
\begin{equation}
f(r,\theta)
=
\sum_{\abs{k}\leq K}a_k(r)e^{ik\theta}
\end{equation}
satisfies the regularity assumptions if each~$a_k$ is H\"older continuous.
(See theorem~\ref{thm:brt}.)
\item
Assuming a~$C^{1,1}$ wave speed satisfying the Herglotz condition and the countable conjugacy condition, the periodic broken ray determines the even part of the function in three or more dimensions.
In two dimensions averages over circles centered at the origin are determined.
Very little other information can be recovered in any dimension.
(See theorem~\ref{thm:pbrt}.)
\end{itemize}

The tools used in these proofs are similar.
They all rely on analysis disc by disc and Abel transforms.
The piecewise~$C^{1,1}$ metric is dealt with using layer stripping.
Fourier series in the angular variable(s) plays a crucial role.

The periodic broken ray transform is closely related to a new integral transform we introduce, the planar average ray transform.
The kernel of this new transform is precisely the space of odd functions (see theorem~\ref{thm:part}), and the mentioned result for the periodic broken ray transform follows.
The proof of this theorem again relies on Abel transforms, but also on the Funk transform.

In order to prove the mentioned results for ray transforms, we give several results for various other integral transforms in a low regularity setting.
For injectivity of Funk-type transforms on distributions, see theorem~\ref{thm:funk}.

In addition, we prove that a large class of Abel-type integral transforms is injective (see theorem~\ref{thm:inj}).
The proof is based on finding a local approximate inverse and using Neumann series together with layer stripping.
Section~\ref{sec:abel} contains several results regarding continuity, injectivity and differentiability properties of these integral transforms.

\subsection{Applications and related results}
\label{sec:app}

The most important spherically symmetric manifold we have in mind is the Earth.
In the Preliminary Reference Earth Model (PREM) both pressure and shear wave speeds are piecewise nice and satisfy the Herglotz condition in the mantle~\cite{DA:prem}.
The Herglotz condition is violated at the core--mantle boundary (CMB) but is satisfied at all jump discontinuities above it.
Furthermore, the Herglotz condition for the pressure wave speed only fails at the CMB and for shear wave speed only in the liquid outer core.
Geometrically, failure of the Herglotz condition at a jump discontinuity means that rays can be trapped by total internal reflection.

We consider a piecewise~$C^{1,1}$ satisfying the Herglotz condition in an annulus.
In PREM, both wave speeds satisfy these assumptions, both regularity and geometry, in the whole mantle.

The attenuated X-ray transform has a direct application in an imaging method known as SPECT.
It also appears in the anisotropic Calder\'on's inverse boundary value problem~\cite{DKSU:anisotropic,KSU:maxwell}.
For a review of attenuated X-ray tomography we refer to~\cite{F:attenuated-x-ray}.
For other results for attenuated ray transforms on manifolds, see e.g.~\cite{S:attenuated-x-ray,B:hyperbolic,SU:surface,M:attenuated-xrt,MSU:x-ray-cp}.

The standard X-ray transform without attenuation has applications in imaging methods such as CT and PET.
The Euclidean version of this problem is the starting point of the study of inverse problems in integral geometry (see e.g. the classical works~\cite{radon,cormack,book-natterer,book-helgason}), but our main interests are in non-Euclidean geometry.
It appears in the linearizations of some inverse problems, such as the travel time tomography problem where one attempts to recover a manifold from the pairwise distances of boundary points.
Solving the linearized problem leads to results known as boundary rigidity; see~\cite{SU:bdy-rigitity-review} for a review.
See the book~\cite{S:tensor-book} for the geodesic X-ray transform and its applications.

A less studied variant of X-ray tomography is broken ray tomography, where instead of geodesics one considers broken rays which reflect at the boundary of the manifold.
One may either consider broken rays with endpoints in a given subset of the boundary, or periodic broken rays which do not terminate at all.
For recent results in broken ray tomography, we refer to~\cite{I:brt-thesis,I:disk,I:refl,I:bdy-det,I:torus,IS:brt-pde-1obst,H:square,H:brt-flat-refl,eskin}.
``Broken ray tomography'' may also refer to problems where the reflections take place in the interior rather than the boundary of the domain; see~\cite{I:brt-thesis} for a discussion of terminology.

The broken ray transform appears in similar ways as the X-ray transform when reflections from the exterior boundary or inaccessible interior boundaries are present.
For inverse boundary value problems for PDEs with partial data, we refer to~\cite{eskin,KS:calderon}.
Differentiating the length of a broken ray with respect to the metric leads to the broken ray transform~\cite[Theorem~17]{IS:brt-pde-1obst} in analogy to the more familiar situation without reflections.

Seismic travel time tomography for USArray~\cite{usarray2008} (there are numerous subsequent improvements as more data have become available) is based on broken ray tomography.
We will study broken ray tomography with measurements in a small boundary set in section~\ref{sec:brt}.
The prevalent inversion method is to make a least squares fit, but a better understanding of broken ray tomography is likely to lead to advances in seismic tomographic inversion.
For a comparison of linearized tomographic travel time inversion and the non-linear Herglotz--Wiechert inversion method in radial symmetry, and a discussion of Abel transforms in this context, see~\cite{N:tomography-herglotz-wiechert}.
Earlier work on the Abel transform will be discussed in section~\ref{sec:abel}.

The periodic X-ray transform on closed manifolds is related to spectral rigidity problems.
This relation is encoded by the length spectrum, the set of lengths of all periodic geodesics.
For a prominent example, see~\cite{PSU:anosov}.
In a similar fashion, the periodic broken ray transform is related to the rigidity of the spectrum and the length spectrum of a manifold with boundary.
The first results in this direction can be found in~\cite{HIK:spherical-spectral} where the problem is studied on spherically symmetric manifolds.

\subsection{The structure of the paper}

Sections~\ref{sec:abel}--\ref{sec:funk} are dedicated to tools we need for ray transforms.
This includes Abel transforms (section~\ref{sec:abel}), Funk transforms (section~\ref{sec:funk}), and other tools (section~\ref{sec:tools}).
A reader whose sole interest lies in ray transforms may skip these sections.
Some of our auxiliary results are well known, but we present them with proofs for completenees and ease of reading.

We will discuss geodesics and broken rays in radial symmetry in section~\ref{sec:geod}.
With the key tools ready for our disposal, we will focus on ray transforms in sections~\ref{sec:xrt}--\ref{sec:pbrt}.

\section{Abel transforms}
\label{sec:abel}

We start with studying Abel transforms.
We will need to understand these transforms to prove injectivity results for ray transforms in subsequent sections, but we look at more general Abel transforms than is needed for ray tomography.
For a review of the Abel transform, see~\cite{VT:abel-review}, and for more about general Abel-like transforms, see \cite{MT:calderon-abel,GV:abelbook,I:disk,cormack,V:volterra1,V:volterra2,V:volterra-spline,V:volterra-abstract}.
We are not sure of originality of the results in this section but we were unable to find the exact results we need in the literature.
The case of $K\equiv1$ is well known.

The notations of this section are somewhat different from those elsewhere in this paper.

\subsection{Definitions}

Let $\Delta=\{(x,y);0\leq x\leq y\leq 1\}$ and let $K\colon\Delta\to\R$ be any bounded measurable function.
Fix a number $\alpha\in[0,1)$.
For a function $f\colon[0,1]\to\R$ we define its integral transform $I^\alpha_K f\colon[0,1]\to\R$ by
\begin{equation}
I^\alpha_K f(x)
=
\int_x^1(y-x)^{-\alpha}K(x,y)f(y)\der y
\end{equation}
whenever this integral is defined, and we let $I^\alpha_K f(1)=0$.
These transforms generalize the classical Abel transform~$I^{1/2}_1$, and they can be seen as weighted Riesz potentials.
We wish to study the properties of this integral transform for different choices of domain and target spaces.
In particular, we are interested in injectivity and continuity.

An important property of the transform is that~$I^\alpha_K f(x)$ only depends on the values of~$f(y)$ for $y\geq x$.
Without this property an injectivity result like theorem~\ref{thm:inj} would not be possible.
From the point of view of ray transforms, this is related to support theorems.

All function spaces are based on the interval~$[0,1]$ with the usual metric and measure unless otherwise stated.
All spaces we consider are contained in~$L^1$, so any function may be assumed to be integrable.
The exponents~$p$ in~$L^p$ spaces may be anything in~$[1,\infty]$, and we denote the H\"older conjugate by a prime: $p'=1/(1-1/p)$.

The kernel~$K$ is often assumed to be in the space~$\Lip(\Delta)$ of Lipschitz functions from~$\Delta$ to~$\R$, but some results hold for $K\in L^\infty(\Delta)$ as well.
We denote the Lipschitz constant with respect to the first variable by~$\lip_1(K)$ for~$K\in\Lip(\Delta)$.

\subsection{Continuity}

We establish continuity in two senses: the transform is continuous between suitable function spaces, and the image of a continuous function --- and some~$L^p$ functions --- is continuous.

\begin{theorem}
\label{thm:abel-cont}
The transform $I^\alpha_K\colon L^p\to L^q$ is well defined and continuous when $\alpha+1/p<1+1/q$.
In particular, this holds when $p>1/(1-\alpha)$, $q<1/\alpha$, or $p=q$.
The norm of this mapping satisfies $\aabs{I^\alpha_K}_{L^p\to L^q}=\Order(\sup_\Delta\abs{K})$.
\end{theorem}

\begin{proof}
We denote by~$\lesssim$ inequalities involving constants independent of~$f$.
If~$K$ vanishes identically, the result is obvious.
Otherwise we divide~$K$ by~$\sup_\Delta\abs{K}$ so that we may assume $\abs{K}\leq1$; the last part of the claim follows from this scaling.

We set $s=\min(p,q)$; this number satisfies $\alpha+1/s<1+1/q$, $q/s\geq1$ and $1\leq s\leq p$.
The inequality $\alpha+1/s<1+1/q$ implies that $\alpha<1/q+1/s'$, and so there are constants $\beta,\gamma\in[0,\alpha]$ so that $s'\gamma<1$, $q\beta<1$, and $\beta+\gamma=\alpha$.

We have
\begin{equation}
\begin{split}
\aabs{I^\alpha_Kf}_{L^q}
&=
\left(
\int_0^1\abs{\int_z^1 f(r)K(z,r)(r-z)^{-\alpha}\der r}^q\der z
\right)^{1/q}
\\&\leq
\left(
\int_0^1\left(\int_z^1 \abs{f(r)}(r-z)^{-\beta}(r-z)^{-\gamma}\der r\right)^q\der z
\right)^{1/q}
\\&\leq
\bigg(
\int_0^1
\left(\int_z^1 \abs{f(r)}^s(r-z)^{-s\beta}\der r\right)^{q/s}
\\&\quad\times
\left(\int_z^1 (r-z)^{-s'\gamma}\der r\right)^{q/s'}
\der z
\bigg)^{1/q}
\\&\lesssim
\left[
\left(
\int_0^1
\left(\int_z^1 \abs{f(r)}^s(r-z)^{-s\beta}\der r\right)^{q/s}
\der z
\right)^{s/q}
\right]^{1/s}
\\&\leq
\left[
\int_0^1
\left(
\int_0^r \abs{f(r)}^q(r-z)^{-q\beta}\der z
\right)^{s/q}
\der r
\right]^{1/s}
\\&\lesssim
\left[
\int_0^1
\abs{f(r)}^s
\der r
\right]^{1/s}
\\&=
\aabs{f}_{L^s}
\\&\leq
\aabs{f}_{L^p}
.
\end{split}
\end{equation}
This is the desired continuity estimate.
\end{proof}

\begin{theorem}
\label{thm:If-cont}
If~$K$ is continuous and $p\in(1,\infty]$ is such that $\alpha p/(p-1)<1$, then $I_K^\alpha\colon L^p\to C$ is continuous.
\end{theorem}

\begin{proof}
It suffices to show that for every $f\in L^p$ the function~$I_K^\alpha f$ is continuous.
The space of continuous functions is a closed subspace of~$L^\infty$ and theorem~\ref{thm:abel-cont} gives continuity in the~$L^\infty$ norm.

We denote $g=I_K^\alpha f$.
Let us first prove that~$g$ is continuous at~$1$.
It is a natural interpretation of the definition of the integral transform that $g(1)=0$.
We can then use H\"older's inequality to find
\begin{equation}
\begin{split}
\abs{g(x)}
&\leq
\int_x^1(y-x)^{-\alpha}\abs{K(x,y)}\abs{f(y)}\der y
\\&\leq
\max_\Delta\abs{K}
\left(\int_x^1(y-x)^{-\alpha p'}\der y\right)^{1/p'}
\left(\int_x^1\abs{f(y)}^p\der y\right)^{1/p}
\\&\to
0
\quad\text{as }x\to1,
\end{split}
\end{equation}
which establishes continuity at $x=1$.

Fix any $x\in[0,1)$.
We will show that $\abs{g(x')-g(x)}\to0$ as $x'\to x$.
We can assume that $x,x'\in[0,a]$ for some $a<1$.

By definition we have
\begin{equation}
g(x')
=
\int_{x'}^1(y'-x')^{-\alpha}K(x',y')f(y')\der y'.
\end{equation}
Making the change of variable $y=\frac{1-x}{1-x'}y'+\frac{x-x'}{1-x'}$ in the above integral gives
\begin{equation}
\begin{split}
g(x')
&=
\left(\frac{1-x'}{1-x}\right)^{1-\alpha}\int_x^1(y-x)^{-\alpha}K\left(x',y+\frac{(x'-x)(1-y)}{1-x}\right)
\\&\quad\times
f\left(y+\frac{(x'-x)(1-y)}{1-x}\right)\der y.
\end{split}
\end{equation}
Therefore
\begin{equation}
\label{eq:abel-cont-est}
\begin{split}
&\abs{g(x')-g(x)}
\\&\leq
\abs{1-\left(\frac{1-x'}{1-x}\right)^{1-\alpha}}\times
\int_x^1(y-x)^{-\alpha}
\\&\qquad\times
\abs{K\left(x',y+\frac{(x'-x)(1-y)}{1-x}\right)}
\abs{f\left(y+\frac{(x'-x)(1-y)}{1-x}\right)}\der y
\\&\quad+
\int_x^1(y-x)^{-\alpha}\abs{K\left(x',y+\frac{(x'-x)(1-y)}{1-x}\right)-K(x,y)}
\\&\qquad\times
\abs{f(y+\frac{(x'-x)(1-y)}{1-x})}\der y
\\&\quad+
\int_x^1(y-x)^{-\alpha}\abs{K(x,y)}\abs{f\left(y+\frac{(x'-x)(1-y)}{1-x}\right)-f(y)}\der y.
\end{split}
\end{equation}
We will show that each of the three terms tends to zero as $x'\to x$.

Let us define $f_{x'}\colon [0,1]\to\mathbb R$ by $f_{x'}(y)=f(y+\frac{(x'-x)(1-y)}{1-x})$; notice that $f_x=f$.
We extend~$f$ by zero to a larger interval $J=[-\frac1{1-a},1]$ to make each~$f_{x'}$ well defined and we also extend~$f_{x'}$ to~$J$ by zero where the original formula would take the argument of~$f$ outside~$J$.
Let us show that $f_{x'}\to f$ in~$L^p([0,1])$ as $x'\to x$.

Fix any $\eps>0$.
By density of continuous functions, there is a continuous function $h\in L^p(J)$ so that $\aabs{f-h}_{L^p(J)}<(1-a)\eps/3$.
Let us define~$h_{x'}$ in terms of~$h$ like~$f_{x'}$ was defined in terms of~$f$.
Since these scaled and shifted functions differ from the original ones only by an affine change of variable, we get
\begin{equation}
\aabs{f_{x'}-h_{x'}}_{L^p([0,1])}
\leq
\frac{1-x'}{1-x}\aabs{f-h}_{L^p(J)}
\leq
\frac\eps3.
\end{equation}
The function~$h$ is uniformly continuous, so there is $\delta>0$ so that $\abs{h(y)-h_{x'}(y)}<\eps/3$ for all $y\in[0,1]$ and all such~$x'$ that $\abs{x-x'}<\delta$.
Now if $\abs{x-x'}<\delta$, we have
\begin{equation}
\begin{split}
\aabs{f_{x'}-f}_{L^p([0,1])}
&\leq
\aabs{f_{x'}-h_{x'}}_{L^p([0,1])}+\aabs{h_{x'}-h}_{L^p([0,1])}
\\&\quad
+\aabs{h-f}_{L^p([0,1])}
\\&<
\eps.
\end{split}
\end{equation}
This proves the desired auxiliary result: $f_{x'}\to f$ in $L^p([0,1])$ as $x'\to x$.

This implies that the functions~$f_{x'}$ have a uniform bound on their~$L^p$ norm for~$x'$ sufficiently close to~$x$.
Using H\"older's inequality like above, we find
\begin{equation}
\begin{split}
&\int_x^1(y-x)^{-\alpha}\abs{K\left(x',y+\frac{(x'-x)(1-y)}{1-x}\right)}\abs{f_{x'}(y)}\der y
\\&\leq
\max_\Delta\abs{K}
\left(\int_x^1(y-x)^{-\alpha p'}\der y\right)^{1/p'}
\left(\int_x^1\abs{f_{x'}(y)}^p\der y\right)^{1/p}.
\end{split}
\end{equation}
Since this bound is independent of~$x'$, the first term in~\eqref{eq:abel-cont-est} tends to zero as $x'\to x$.

Essentially the same estimate together with uniform continuity of~$K$ shows that the second term tends to zero as well.
The same estimate works for the third term too, since $f_{x'}\to f$ in~$L^p([0,1])$.
\end{proof}

We will also need a Lipschitz version of the previous theorem.

\begin{theorem}
\label{thm:If-lip}
If~$f$ and~$K$ are continuous and $\alpha\in[0,1)$, then~$I_K^\alpha f$ is continuous.
If $\alpha>0$, and~$f$ and~$K$ are additionally Lipschitz, then~$I_K^\alpha f$ is $(1-\alpha)$-H\"older everywhere and locally Lipschitz on~$[0,1)$.
In particular, the local Lipschitz constant of~$I_K^\alpha f$ at $x\in[0,1)$ is at most
\begin{equation}
\sup_\Delta\abs{K}\lip f+\sup_{[0,1]}\abs{f}\lip_2 K
+\frac{1-\alpha}{1-x}\sup_\Delta\abs{K}\sup_{[0,1]}\abs{f}.
%
\end{equation}
\end{theorem}

The estimate for the local Lipschitz constant implies, in particular, that the Lipschitz constant of the restriction $I_K^\alpha f|_{[0,1-\eps]}$ is at most
\begin{equation}
\sup_\Delta\abs{K}\lip f+\sup_{[0,1]}\abs{f}\lip_2 K
+\frac{1-\alpha}{\eps}\sup_\Delta\abs{K}\sup_{[0,1]}\abs{f}.
\end{equation}

\begin{proof}[Proof of theorem~\ref{thm:If-lip}]
We denote $g=I_K^\alpha f$ and prove the first claim first.
By definition we have
\begin{equation}
g(x')
=
\int_{x'}^1(y'-x')^{-\alpha}K(x',y')f(y')\der y'
\end{equation}
for any $x'\in[0,1)$.
Take any $x\in[0,1)$.
Making the change of variable $y=\frac{1-x}{1-x'}y'+\frac{x-x'}{1-x'}$ in the above integral gives
\begin{equation}
\label{eq:vv3}
\begin{split}
g(x')
&=
\left(\frac{1-x'}{1-x}\right)^{1-\alpha}\int_x^1(y-x)^{-\alpha}
\\&\quad\times
K(x',y+\frac{(x'-x)(1-y)}{1-x})f(y+\frac{(x'-x)(1-y)}{1-x})\der y.
\end{split}
\end{equation}
Since~$f$ and~$K$ are uniformly continuous, it is easy to see that $g(x')\to g(x)$ as $x'\to x$ with a H\"older modulus of continuity.

We have thus shown that~$g$ is continuous on~$[0,1)$, so it remains to show that $g(x)\to 0$ as $x\to 1$.
But this is elementary, as~$f$ and~$K$ are bounded.

The second claim follows easily from expression~\eqref{eq:vv3} for~$g(x')$.
\end{proof}

The second claim of theorem~\ref{thm:If-lip} cannot be improved significantly: if $K\equiv1-\alpha$ and $f\equiv1$, then $I_K^\alpha f(x)=(1-x)^{1-\alpha}$.

\subsection{Injectivity}

In this section we assume that $\alpha\in(0,1)$.

To study injectivity of the Abel transform~$I^\alpha_K$, we define the integral
\begin{equation}
\label{eq:J-def}
J_K^\alpha(x,y)
=
\int_x^y(z-x)^{\alpha-1}(y-z)^{-\alpha}K(z,y)\der z,
\end{equation}
which will turn out to be of great use.

A simple application of Fubini's theorem shows that
\begin{equation}
\label{eq:fubini}
I_1^{1-\alpha} I_K^\alpha=I_{J_K^\alpha}^0.
\end{equation}
We will show injectivity of~$I_K^\alpha$ under some conditions by showing that~$I_{J_K^\alpha}^0$ is injective.

Another important property of the integral~$J_K^\alpha$ is that
\begin{equation}
\label{eq:c-def}
J_1^\alpha(x,y)
=
\frac{\pi}{\sin(\alpha\pi)}
\eqqcolon
c_\alpha
\end{equation}
whenever $x<y$.
This gives the following result; see~\cite{cormack} for an early use of such identities in ray tomography.

\begin{proposition}
\label{prop:K1-inj}
The mapping $I_1^\alpha\colon L^p\to L^q$ is injective for any $\alpha\in(0,1)$ whenever~$p$ and~$q$ are such that it is well defined.
We have the inversion formula
\begin{equation}
\label{eq:K1-inv}
f(x)=-c_\alpha^{-1}\Der{x}I_1^{1-\alpha} I_1^\alpha f(x)
\end{equation}
which holds almost everywhere.
The formula holds pointwise for continuous~$f$.
\end{proposition}

\begin{proof}
For $f\in L^1$ we have by~\eqref{eq:fubini} and~\eqref{eq:c-def}
\begin{equation}
I_1^{1-\alpha} I_1^\alpha f(x)
=
c_\alpha \int_x^1 f(y)\der y,
\end{equation}
from which the result follows.
The required version of the fundamental theorem of calculus for~$L^1$ functions follows from Lebesgue's differentiation theorem.
\end{proof}

The general transform~$I_K^\alpha$ can be inverted in a similar manner if one only finds such a kernel $L\colon\Delta\to\R$ that $I_L^{1-\alpha} I_K^\alpha=I_1^0$.
When $\alpha=1/2$ and $K(x,y)=2yT_k(x/y)/\sqrt{x+y}$, the choice $L(x,y)=T_k(y/x)x/(\pi y\sqrt{x+y})$ gives this result; the corresponding integral transform~$I_K^{1/2}$ is closely related to the Radon transform in the plane~\cite{cormack,I:disk}.
Here~$T_k$ is the $k$th Chebyshev polynomial.
The apparent singularity at zero is not an issue for injectivity; theorem~\ref{thm:inj} works so that any neighborhood of zero can be easily omitted.

For general~$K$ and~$\alpha$ it is difficult to find a suitable~$L$, so we approach the problem in a different way.
The downside of the method below is that it does not yield an explicit inversion formula like~\eqref{eq:K1-inv}.
The result of proposition~\ref{prop:K1-inj} can be slightly generalized as the next proposition demonstrates.

\begin{proposition}
\label{prop:fact-inj}
Let $K\colon\Delta\to\R$ be a continuous nowhere vanishing function so that $K(x,y)=a(x)b(y)$ for some functions~$a$ and~$b$.
Then~$I_K^\alpha$ is injective and
\begin{equation}
\label{eq:fact-inv}
f(x)=-b(x)^{-1}c_\alpha^{-1}\Der{x}I_L^{1-\alpha} I_K^\alpha f(x),
\end{equation}
where $L(x,y)=1/a(y)$.
\end{proposition}

\begin{proof}
A simple calculation shows that $I_L^{1-\alpha}I_K^\alpha f=c_\alpha I_1^0 (bf)$, from which the result follows.
\end{proof}

Let~$A$ be the space of absolutely continuous real functions on~$[0,1]$ which vanish at~$1$.
It is a classical result (see e.g.~\cite[Chapter 3, Section 3.2]{SS:real-analysis}) that $I_1^0(L^1)=A$, so that we may equip~$A$ with the norm that makes $I_1^0\colon L^1\to A$ an isometry.

For $\eta\in[0,1)$ we will use the subscript~$\eta$ in a function space to indicate restriction to~$[\eta,1]$ instead of~$[0,1]$.
Thus we write $L^p_\eta=L^p([\eta,1])$ and~$A_\eta$ for absolutely continuous functions on~$[\eta,1]$ vanishing at~$1$.
We write $\Delta_\eta=\{(x,y);\eta\leq x\leq y\leq 1\}$.

We define $\phi^\alpha\colon L^\infty(\Delta)\to\lin(L^1,A)$ by letting $\phi^\alpha(K)=I^{\alpha-1}_1I^\alpha_K$.
Lemma~\ref{lma:phi} below shows that indeed $\phi^\alpha(K)(L^1)\subset A$, so the mapping~$\phi^\alpha$ is well defined.

\begin{lemma}
\label{lma:J}
For $K\in\Lip(\Delta)$ also $J^\alpha_K\in\Lip(\Delta)$.
In particular, we have $\sup_\Delta\abs{J^\alpha_K}\leq c_\alpha\sup_\Delta\abs{K}$ and $\lip_1(J^\alpha_K)\leq c_\alpha\lip_1(K)$.
\end{lemma}

\begin{proof}
It follows immediately from~\eqref{eq:J-def} and~\eqref{eq:c-def} that
\begin{equation}
\sup_\Delta\abs{J^\alpha_K}
\leq
c_\alpha\sup_\Delta\abs{K}.
\end{equation}
The change of variable $z=\frac{y-x}{y-x'}z'+y\frac{x-x'}{y-x'}$ gives
\begin{equation}
\begin{split}
&
J^\alpha_K(x,y)-J^\alpha_K(x',y)
\\&=
\int_x^y(z-x)^{\alpha-1}(y-z)^{-\alpha}
\\&\quad\times
\left[K(z,y)-K\left(\frac{y-x'}{y-x}z+\frac{x-x'}{x-y}y,y\right)\right]\der z.
\end{split}
\end{equation}
Thus
\begin{equation}
\begin{split}
&
\abs{J^\alpha_K(x,y)-J^\alpha_K(x',y)}
\\&\leq
\lip_1(K)\int_x^y(z-x)^{\alpha-1}(y-z)^{-\alpha}\frac{\abs{x-x'}(y-z)}{y-x}\der z
\\&\leq
\lip_1(K)\abs{x-x'}\int_x^y(z-x)^{\alpha-1}(y-z)^{-\alpha}\der z
\\&=
\lip_1(K)\abs{x-x'}c_\alpha
\end{split}
\end{equation}
and so $\lip_1(J^\alpha_K)\leq c_\alpha\lip_1(K)$.
\end{proof}

\begin{lemma}
\label{lma:phi}
The mapping $\phi^\alpha(K)\colon L^1([0,1])\to A([0,1])$ is well defined and continuous for any $\alpha\in(0,1)$ and $K\in L^\infty(\Delta)$.
For $K\in\Lip(\Delta)$, there is $\eta\in[0,1)$ depending only on~$K$ and~$\alpha$ so that $\phi^\alpha_\eta(K)\colon L^1_\eta\to A_\eta$ satisfies $\aabs{\phi^\alpha_\eta(K)}\leq 2c_\alpha\sup_{\Delta_\eta}\abs{K}$.
This estimate remains true if~$\eta$ is increased.
We also have the estimate $\aabs{\phi^\alpha_\eta}\geq c_\alpha$ for any~$\eta$.
\end{lemma}

\begin{proof}
We have $\phi^\alpha(K)=I_{J_K^\alpha}^0$ by equation~\eqref{eq:fubini}.
Let $0\leq x<y\leq1$ and $f\in L^1$.
For short, we write $g=\phi^\alpha(K)f$.
We have
\begin{equation}
\begin{split}
\abs{g(1-\eps)}
&=
\abs{\int_{1-\eps}^1J_K^\alpha(1-\eps,z)f(z)\der z}
\\&\leq
\sup_\Delta \abs{J_K^\alpha} \int_{1-\eps}^1\abs{f(z)}\der z
\\&\to
0\quad\text{as }\eps\to0
\end{split}
\end{equation}
by absolute continuity of the Lebesgue integral.
Once we establish absolute continuity, this shows that $g(1)=0$.

For absolute continuity we have the estimate
\begin{equation}
\label{eq:ac-est}
\begin{split}
\abs{g(x)-g(y)}
&\leq
\sup_\Delta\abs{J_K^\alpha}\int_x^y\abs{f(z)}\der z
\\&\quad
+\lip_1(J_K^\alpha)(y-x)\aabs{f|_{[y,1]}}_{L^1}.
\end{split}
\end{equation}
Let $\eps>0$; we wish to find $\delta>0$ so that
\begin{equation}
\sum_{i\in\N}\abs{g(a_i)-g(b_i)}\leq\eps
\end{equation}
whenever the intervals~$(a_i,b_i)$ are disjoint and their union has measure at most~$\delta$.
But this follows now immediately from the estimate~\eqref{eq:ac-est}, integrability of~$f$ and absolute continuity of the Lebesgue integral.
We have now proven that $g=\phi^\alpha(K)f\in A$.

Since $g\in A$, its derivative~$g'$ exists almost everywhere and is in~$L^1$.
We denote $M=\sup_\Delta\abs{J_K^\alpha}$ and $L=\lip_1(J_K^\alpha)$.
It follows from lemma~\ref{lma:J} that $M\leq c_\alpha\sup_\Delta\abs{K}$ and $L<\infty$.
Estimate~\eqref{eq:ac-est} shows that
\begin{equation}
\abs{g'(x)}
\leq
M\abs{f(x)}
+
L\int_x^1\abs{f(z)}\der z
\end{equation}
for almost all~$x$.

Thus for any $\eta\in[0,1)$
\begin{equation}
\begin{split}
\aabs{g}_{A([\eta,1])}
&=
\aabs{g'}_{L^1([\eta,1])}
\\&\leq
[M+L(1-\eta)]\aabs{f}_{L^1([\eta,1])}.
\end{split}
\end{equation}
This shows that $\aabs{\phi^\alpha_\eta(K)}\leq M+L(1-\eta)$.
If we choose $\eta=\max(0,1-M/L)$, we have $\aabs{\phi^\alpha_\eta(K)}\leq2M$ as claimed.
Since all variables in~$\Delta$ take values in~$\Delta_\eta$ only, we may replace supremum over~$\Delta$ with supremum over~$\Delta_\eta$.

For a lower bound on the operator norm we use the constant kernel $K\equiv1$.
We have
\begin{equation}
\aabs{\phi^\alpha_\eta(1)f}_{A_\eta}
=
c_\alpha\aabs{I_1^0f}_{A_\eta}
=
c_\alpha\aabs{f}_{L^1_\eta},
\end{equation}
so that $\aabs{\phi^\alpha_\eta(1)}\geq c_\alpha$ and thus $\aabs{\phi^\alpha_\eta}\geq c_\alpha$.
\end{proof}


\begin{lemma}
\label{lma:neumann}
Let $\alpha\in(0,1)$.
Let $K\in\Lip(\Delta)$ be nonvanishing at the diagonal.
There is $\eta\in[0,1)$ depending only on~$\alpha$, $\inf_x\abs{K(x,x)}$, $\sup_\Delta\abs{K}$, and~$\lip_1(K)$ so that $I^\alpha_K\colon L^1_\eta\to L^1_\eta$ is injective.
The result remains true if~$\eta$ is increased.
\end{lemma}

\begin{proof}
Since~$K$ cannot change sign at the diagonal, we may assume $m\coloneqq\inf_x K(x,x)>0$.
Let us denote by~$C$ the constant function $C\equiv K(1,1)\geq m$ on~$\Delta$.
By lemma~\ref{lma:phi} we can choose~$\eta$ so that $\aabs{\phi^\alpha_\eta(C-K)}\leq 2\sup_{\Delta_\eta}\abs{C-K}$.

Furthermore we may take~$\eta$ to be so large that $\sup_{\Delta_\eta}\abs{C-K}<Cc_\alpha/2$, so that $\aabs{\phi^\alpha_\eta(C-K)}<Cc_\alpha$.
This choice of~$\eta$ indeed only depends on~$\alpha$, $m$, $\sup_\Delta\abs{K}$, and~$\lip_1(K)$.

It suffices to show that $D\phi^\alpha_\eta(K)\colon L^1_\eta\to L^1_\eta$ is bijective, where~$D$ is the derivative operator.
We then define the operators $E=DI^{\alpha-1}_1I^\alpha_C$ and $F=DI^{\alpha-1}_1I^\alpha_{C-K}$, so that $D\phi^\alpha_\eta(K)=E-F$.
We do not include the subscript~$\eta$ in the operators~$E$ and~$F$, although they map~$L^1_\eta$ to itself.
Since $E=-Cc_\alpha\id$, we have $\aabs{E^{-1}}=(Cc_\alpha)^{-1}$.
Thus
\begin{equation}
\aabs{F}
=
\aabs{\phi^\alpha(C-K)}
<
Cc_\alpha
=
\aabs{E^{-1}}^{-1}.
\end{equation}
Since~$\lin(L^1,L^1)$ is a Banach algebra, we may use the Neumann series to invert $D\phi^\alpha(K)=E-F$:
\begin{equation}
\phi^\alpha(K)^{-1}
=
E^{-1}\sum_{n=0}^\infty (FE^{-1})^n.
\end{equation}
This concludes the proof.
\end{proof}

\begin{theorem}
\label{thm:inj}
Let $\alpha\in[0,1)$.
Suppose $K\colon\Delta\to\R$ is bounded everywhere, nonzero on the diagonal $\{(x,x);0\leq x\leq 1\}$, and Lipschitz continuous in some neighborhood of the diagonal.
If $f\in L^1([0,1])$ satisfies $I^\alpha_K f(x)=0$ for almost all $x\geq r$ for some $r\in[0,1)$, then $f(x)=0$ for almost all $x\geq r$.
In particular, $I^\alpha_K\colon L^1\to L^1$ is injective.
\end{theorem}

\begin{remark}
One can write an inversion formula for theorem~\ref{thm:inj} in terms of Neumann series and iteration as may be read in the proof.
\end{remark}

\begin{remark}
The restriction result of theorem~\ref{thm:inj} (if~$I^\alpha_K f$ vanishes above~$r$, so does~$f$) is related to support theorems for ray transforms.
Indeed, this result combined with the analysis of~\cite{cormack} proves Helgason's famous support theorem.
Helgason~\cite[Theorem~4.2]{book-helgason} gave a different proof.
More generally, this result gives a support theorem for ray transforms on spherically symmetric manifolds; see section~\ref{sec:xrt}.
\end{remark}

\begin{proof}[Proof of theorem~\ref{thm:inj}]
We will only prove the theorem for $\alpha>0$.
If $\alpha=0$, we may consider~$I_K^0$ directly instead of studying~$I_{J^0_K}^0$.
Apart from removing this one step the proof is unchanged.

The kernel~$K$ is strictly positive and Lipschitz continuous in $\{(x,y)\in\Delta;y-x\geq\zeta\}$ for some $\zeta\in(0,1]$.
We will only use the values of~$K$ in this strip near the diagonal, so we may replace~$K$ with a Lipschitz extension elsewhere without altering the result and assume that~$K$ is Lipschitz in~$\Delta$.

Let $\eta\in[0,1)$ be the constant of lemma~\ref{lma:neumann} related to~$K$ and~$\alpha$.
We denote $\delta=1-\eta$; by possibly slightly decreasing~$\delta$, we may assume that $(1-r)/\delta\eqqcolon n$ is an integer.
Lemma~\ref{lma:neumann} remains true if~$\delta$ is decreased and~$\eta$ thus increased.

The function $f_\eta=f|_{[\eta,1]}\in L^1_\eta$ satisfies $\phi^\alpha_\eta(K)f_\eta=0$, and by lemma~\ref{lma:neumann} this implies that $f_\eta=0$.
Therefore~$f$ vanishes almost everywhere in $[1-\delta,1]$.
If $n=1$ ($\eta=r$), we are done, so we assume that $n\geq2$.

We then define $g\in L^1([0,1])$ by letting $g(x)=0$ for $x<\delta$ and $g(x)=f(x-\delta)$ for $x\geq\delta$.
We define the kernel $L\colon\Delta\to\R$ by $L(x,y)=K(x-\delta,y-\delta)$ when $x\geq\delta$ and extend~$L$ to the rest of~$\Delta$.
(We never use the extended values of~$L$.)
This extension can be done so that the maximum and Lipschitz constant of~$L$ are at most those of~$K$ and the infimum on the diagonal is not decreased.
Since~$f$ vanishes on $[1-\delta,1]$, a simple change of variable gives
\begin{equation}
I_L^\alpha g(x)=I_K^\alpha f(x-\delta)
\end{equation}
for all $x\geq r+\delta$.
If additionally $x\geq 1-2\delta$, lemma~\ref{lma:neumann} may again be used to see that~$g$ vanishes on~$[\eta,1]$.
Note that the step size~$\delta$ can be kept constant since~$\eta$ in lemma~\ref{lma:neumann} only depends on the bounds and the Lipschitz constant of~$K$ which are not changed.
Thus we have shown that~$f$ vanishes on $[1-2\delta,1]$.

We can carry on inductively~$n$ times (recall that $1-n\delta=r$), and finally conclude that~$f$ vanishes on~$[r,1]$.
\end{proof}

\subsection{Differentiability}

So far in this section we have studied Abel transforms with very low regularity.
In this subsection we will study how Abel transforms preserve differentiability.

In particular, we study functions of the form
\begin{equation}
\label{eq:abel-diff}
f(x)
=
\int_x^1(y^2-x^2)^{-\alpha}\phi(x,y)\der y,
\end{equation}
where~$\phi$ is a regular function.
The question is how much regularity~$f$ inherits from~$\phi$, and an answer is given by the following proposition.
The proposition will not be used in this paper directly, but it is employed in the follow-up work~\cite{HIK:spherical-spectral} and is used in the proof of proposition~\ref{prop:ccc-dense}.

\begin{proposition}
\label{prop:abel-diff}
Suppose $\phi\colon\Delta\to\R$ is continuous and~$k$ times continuously differentiable in the interior of~$\Delta$.
Provided that $\alpha\in(0,1)$, the function~$f$ defined by~\eqref{eq:abel-diff} is continuous in~$[0,1]$ and~$k$ times continuously differentiable in~$(0,1)$.
Furthermore,
\begin{equation}
\label{eq:abel-derivative}
\begin{split}
f'(x)
&=
\int_x^1(y^2-x^2)^{-\alpha}\left[\partial_x\phi(x,y)+\partial_y\left(\frac xy\phi(x,y)\right)\right]\der y
\\&\quad
-x(1-x^2)^{-\alpha}\phi(x,1)
.
\end{split}
\end{equation}
\end{proposition}

\begin{proof}
Continuity follows from theorem~\ref{thm:If-lip}, so we only prove differentiability.

Fix $x\in(0,1)$ and let $\tilde x\in(0,1)$.
We will eventually let $\tilde x\to x$ and show that the difference quotient has the correct limit.
Let $k=1$ first.

Adding tildes to~\eqref{eq:abel-diff}, we have
\begin{equation}
f(\tilde x)
=
\int_{\tilde x}^1(\tilde y^2-\tilde x^2)^{-\alpha}\phi(\tilde x,\tilde y)\der\tilde y.
\end{equation}
We change the integration variable from~$\tilde y$ to~$y$ so that $y^2-x^2=\tilde y^2-\tilde x^2$.
We write~$\tilde y$ instead of the radical $\sqrt{y^2+\tilde x^2-x^2}$ for the sake of brevity and legibility.
We find
\begin{equation}
\label{eq:abel-diff-cov}
f(\tilde x)
=
\int\limits_x^{\mathclap{\sqrt{1+x^2-\tilde x^2}}}(y^2-x^2)^{-\alpha}\phi(\tilde x,\tilde y)\frac{y}{\tilde y}\der y.
\end{equation}
Simple calculations give the following approximations:
\begin{equation}
\begin{split}
\tilde y-y
&=
\frac{\tilde x+x}{2y}(\tilde x-x)
+\Order((\tilde x-x)^2)
,
\\
\sqrt{1+x^2-\tilde x^2}
&=
1-\frac12(\tilde x+x)(\tilde x-x)
+\Order((\tilde x-x)^2)
,
\\
\frac{y}{\tilde y}
&=
1-\frac{(\tilde x+x)(\tilde x-x)}{2y^2}
+\Order((\tilde x-x)^2)
,\text{ and}
\\
\phi(\tilde x,\tilde y)
&=
\phi(x,y)+\partial_x\phi(x,y)(\tilde x-x)
\\&\quad+\partial_y\phi(x,y)\frac{\tilde x+x}{2y}(\tilde x-x)
+\order((\tilde x-x)^2).
\end{split}
\end{equation}
Using these, we obtain
\begin{equation}
\begin{split}
\frac{f(\tilde x)-f(x)}{\tilde x-x}
&=
\int_x^1(y^2-x^2)^{-\alpha}
\\&\qquad\times
\left[\partial_x\phi(x,y)+\frac xy\partial_y\phi(x,y)-\frac x{y^2}\phi(x,y)\right]\der y
\\&\quad
-x(1-x^2)^{-\alpha}\phi(x,1)
+
\order(\tilde x-x)
.
\end{split}
\end{equation}
This gives~\eqref{eq:abel-derivative}.

The general case $k>1$ follows easily by induction on~$k$, using~\eqref{eq:abel-derivative}.
The integrands in~\eqref{eq:abel-diff} and~\eqref{eq:abel-derivative} have the same form.
\end{proof}

\begin{remark}
A far simpler calculation verifies that~\eqref{eq:abel-derivative} is also true in the trivial case $\alpha=0$.
\end{remark}

\section{Fourier series and discs in balls}
\label{sec:tools}

This section is devoted to auxiliary results that facilitate connecting the Abel transforms to ray transforms.

\subsection{Measurability and integrability on submanifolds}

The lemma below is given in more generality than needed, but it causes no added effort for the proof.

Let~$G_2^n$ denote the Grassmannian of two-dimensional subspaces of~$\R^n$.
We give the next lemma only for the Euclidean metric, but this causes no loss of generality in the subsequent proofs.
For another version of this statement, see~\cite{S:xrt}.

\begin{lemma}
\label{lma:lp-restriction}
Fix an integer $n\geq3$, a radius $R\in(0,1)$, and an exponent $p\in[1,\infty)$.
Let $A=\bar B(0,1)\setminus\bar B(0,R)\subset\mathbb R^n$.
If $f\in L^p(A)$, then
\begin{equation}
\label{eq:lp-identity}
\int_A\abs{f}^p
=
c\int_{G_2^n}\left(\int_{A\cap P}\abs{x}^{n-2}\abs{f(x)}^p\der\h^2(x)\right)\der P,
\end{equation}
where $c>0$ is a constant and~$\h^2$ is the Hausdorff measure of dimension two.
In particular, $f|_{A\cap P}\in L^p(A\cap P)$ for almost every two dimensional plane $P\in G_2^n$.
\end{lemma}


\begin{proof}
It suffices to show
\begin{equation}
\label{eq:lp-id2}
\int_A f
=
c\int_{G_2^n}\left(\int_{A\cap P}\abs{x}^{n-2}f(x)\der\h^2(x)\right)\der P
\end{equation}
for all $f\in L^1(A)$.
The identity~\eqref{eq:lp-identity} then follows after replacing~$f$ with $\abs{f}^p$.

Let us consider the set
\begin{equation}
E
=
\{(x,v)\in A\times S^{n-1};v\cdot x=0\}
\end{equation}
and the projection $\pi\colon E\to A$ to the first coordinate.
This is a smooth bundle over~$A$ and~$A$ can be viewed as a Riemannian manifold with boundary.
The manifold~$E$ inherits a natural Riemannian metric from the tangent bundle of~$A$.

Since~$E$ has locally a product structure and each fiber~$\pi^{-1}(x)$ has the same measure, we have $\pi^*f\in L^1(E)$ and
\begin{equation}
\label{eq:EA-int}
\int_E \pi^*f
=
\abs{S^{n-2}}\int_Af.
\end{equation}
It will be easier to convert the integral over~$E$ to an integral involving the Grassmannian.

Consider the mapping $F\colon E\to G^n_2$ defined so that~$F(x,v)$ is the unique plane containing both~$x$ and~$v$.
This is a smooth surjection with an everywhere surjective differential, so the smooth coarea formula gives
\begin{equation}
\label{eq:coarea}
\int_E \pi^*f
=
\int_{P\in G^n_2}\int_{F^{-1}(P)}\pi^*f(z)\frac1{NJF(z)}\der\sigma_P(z)\der P,
\end{equation}
where~$\sigma_P$ is the measure on the level set~$F^{-1}(P)$ and~$NJF$ is the normal Jacobian of~$F$.

We could calculate~$NJF(z)$ explicitly, but we can deduce its form without such calculations.
We have $\pi(F^{-1}(P))=P\cap A$ for every $P\in G^n_2$, and due to rotation symmetry we must have
\begin{equation}
\label{eq:AP-int}
\int_{F^{-1}(P)}\pi^*f(z)\frac1{NJF(z)}\der\sigma_P(z)
=
\int_{P\cap A} f(x)\phi(\abs{x})\der\h^2(x)
\end{equation}
for every plane $P\in G^n_2$, where~$\phi$ is some weight function.
Combining equations~\eqref{eq:EA-int}, \eqref{eq:coarea} and~\eqref{eq:AP-int}, we find
\begin{equation}
\abs{S^{n-2}}\int_Af
=
\int_{P\in G^n_2}
\int_{P\cap A} f(x)\phi(\abs{x})\der\h^2(x)
\der P.
\end{equation}
If we let~$f$ be a spherically symmetric function supported near radius~$r$ and observe how both sides of our identity scale as~$r$ varies, we find that $\phi(r)=r^{n-2}$ up to a multiplicative constant.
This proves identity~\eqref{eq:lp-id2}.
\end{proof}

\begin{remark}
Lemma~\ref{lma:lp-restriction} is otherwise valid if we use the ball instead of the annulus (corresponding to $R=0$), but the last statement is false.
That is, equation~\eqref{eq:lp-identity} is valid but the restrictions of an~$L^p$ function need not be in~$L^p$.
The reason is that functions can concentrate near the origin but this is counteracted by the weight $\abs{x}^{n-2}$ in~\eqref{eq:lp-identity}.
For example, the function $f(x)=\abs{x}^{-n+1/2}$ is in $L^1(B(0,1))$ but its restriction to any proper subspace is not integrable.
\end{remark}

\begin{remark}
\label{rmk:sn-lp}
It is not hard to modify the proof given above to show that a function in~$L^p(S^{n-1})$ restricts to an~$L^p$ function on almost every great circle in the sense of the natural measure on the Grassmannian~$G^n_2$.
\end{remark}

\subsection{Angular Fourier series}
\label{sec:fourier}

In this subsection we show that an~$L^2$ function on a spherically symmetric surface can be written as a convergent Fourier series.
It suffices to prove the next lemma in Euclidean geometry.

\begin{lemma}
\label{lma:fourier}
Let $A=\bar B(0,1)\setminus\bar B(0,R)\subset\R^2$ be an annulus.
Any~$L^2$ function $f\colon A\to\C$ can be represented by the series
\begin{equation}
\label{eq:fourier}
f(r,\theta)=\sum_{k\in\Z}a_k(r)e^{ik\theta},
\end{equation}
where $a_k\colon[R,1]\to\C$ are~$L^2$ functions.
This series converges to~$f$ in~$L^2(A)$.

The~$L^2$ norm of~$f$ can be written as
\begin{equation}
\label{eq:fourier-norm}
\aabs{f}_{L^2}^2
=
2\pi\sum_k\int_R^1 \abs{a_k(r)}^2 r\der r.
\end{equation}
\end{lemma}

\begin{proof}
%
We define
\begin{equation}
a_k(r)
=
\frac{1}{2\pi}\int_0^{2\pi}f(r,\theta)e^{-ik\theta}\der\theta.
\end{equation}
Since $f\in L^2(A)$, each~$a_k(r)$ is well defined for almost all $r\in[R,1]$.

To study convergence, we define the truncated series
\begin{equation}
f_m(r,\theta)=\sum_{-m<k<m}a_k(r)e^{ik\theta}
\end{equation}
and an error function
\begin{equation}
g_m(r)=\aabs{f(r,\cdot)-f_m(r,\cdot)}_{L^2(S^1)}^2
\end{equation}
for $m\geq0$.

Since $f\in L^2(A)$, we have $f(r,\cdot)\in L^2(S^1)$ for almost all $r\in[R,1]$.
Thus for almost all~$r$ the Fourier series for~$f(r,\cdot)$ converges in~$L^2(S^1)$, and $0\leq g_m(r)\to0$ monotonically almost everywhere as $m\to\infty$.
Monotonicity implies, in particular, that $g_m\leq g_0$ pointwise.
Since $g_0(r)=\aabs{f(r,\cdot)}_{L^2(S^1)}^2$, we know that $g_0\in L^1([R,1])$.
Therefore we may use the dominated convergence theorem to conclude that
\begin{equation}
\label{eq:fourier-conv}
\begin{split}
\lim_{m\to\infty}\aabs{f-f_m}_{L^2(A)}^2
&=
\lim_{m\to\infty}\int_R^1 g_m(r) r\der r
\\&=
\int_R^1 \lim_{m\to\infty}g_m(r) r\der r
=
0.
\end{split}
\end{equation}
The estimate $g_m\leq g_0\in L^1$ shows that $f_m\in L^2(A)$ for all~$m$, as can be seen in the above calculation.

Now $f_m\in L^2(A)$, so
\begin{equation}
\label{eq:fourier-norm-partial}
\begin{split}
\aabs{f_m}_{L^2}^2
&=
2\pi\sum_{\abs{k}<m}\int_R^1 \abs{a_k(r)}^2 r\der r
<\infty.
\end{split}
\end{equation}
Thus $a_k\in L^2([R,1])$, and the norm representation~\eqref{eq:fourier-norm} follows easily from~\eqref{eq:fourier-conv} and~\eqref{eq:fourier-norm-partial}.
\end{proof}

\section{The Funk transform}
\label{sec:funk}

The Funk transform takes a function on the sphere~$S^n$ to a its integrals over geodesics.
It is a classical result (see lemma~\ref{lma:funk-smooth} below) that the Funk transform of a smooth function vanishes if and only if the function is odd with respect to the antipodal reflection.
We extend this result to distributions --- the result is probably not new but we give an explicit proof as we have been unable to locate one in the literature.

\subsection{Funk transforms on smooth functions}

Let~$G^n_k$ be the Grassmannian of $k$-dimensional subspaces of~$\mathbb R^n$.
We will work with $\tilde G^n_k=G^{n+1}_{k+1}$ instead so that $S^n\cap P$ has dimension~$k$ for all $P\in\tilde G^n_k$.

For $0<k<n$, we define the Funk transform $F^n_k\colon C^\infty(S^n)\to C^\infty(\tilde G^n_k)$ by
\begin{equation}
F^n_kf(P)
=
\fint_{P\cap S^n}f(x)\der x.
\end{equation}
All integrals we encounter in this section come with a natural measure on a sphere or a Grassmannian, and we only take average integrals.
We could include the cases $k=0$ and $k=n$, but these would only correspond to taking the even part and the average of the function.

We will use the integral operators $I^n_{k,l}\colon C^\infty(\tilde G^n_k)\to C^\infty(\tilde G^n_l)$ for $k\leq l$ defined by
\begin{equation}
I^n_{k,l}f(P)
=
\fint_{P\subset Q}f(Q)\der Q.
\end{equation}
This transform is known as the Radon transform on the Grassmannian~\cite{R:grassmann-radon}.
These operators satisfy
\begin{equation}
\label{eq:funk-identities}
\begin{split}
I^n_{k,k}&=\id,\\
I^n_{k,l}\circ I^n_{m,k}&=I^n_{m,l},\quad\text{and}\\
I^n_{k,l}\circ F^n_k&=F^n_l
\end{split}
\end{equation}
whenever the indices are such that the operators are defined.

\subsection{Duality}

The operators introduced above have natural dual operators with respect to~$L^2$ inner products on the manifolds involved.
These are $(F^n_k)^*\colon C^\infty(\tilde G^n_k)\to C^\infty(S^n)$ given by
\begin{equation}
(F^n_k)^*f(x)
=
\fint_{x\in P}f(P)\der P
\end{equation}
and $(I^n_{k,l})^*\colon C^\infty(\tilde G^n_l)\to C^\infty(\tilde G^n_k)$ given by
\begin{equation}
(I^n_{k,l})^*f(P)
=
\fint_{Q\subset P}f(Q)\der Q.
\end{equation}
These will allow defining the transforms~$I^n_{k,l}$ and~$F^n_K$ on distributions by duality.

On a compact manifold~$M$, let us denote the duality pairing between~$C^{-\infty}(M)$ and~$C^\infty(M)$ by~$\ip{\cdot}{\cdot}_M$.
We now define $F^n_k\colon C^{-\infty}(S^n)\to C^{-\infty}(\tilde G^n_k)$ by letting
\begin{equation}
\ip{F^n_kf}{g}_{\tilde G^n_k}
=
\ip{f}{(F^n_k)^*g}_{S^n}
\end{equation}
for all $f\in C^{-\infty}(S^n)$ and $g\in C^\infty(\tilde G^n_k)$.
When restricted to smooth functions, this agrees with the original definition.

Similarly, we define $I^n_{k,l}\colon C^{-\infty}(\tilde G^n_k)\to C^{-\infty}(\tilde G^n_l)$ by letting
\begin{equation}
\ip{I^n_{k,l}f}{g}_{\tilde G^n_l}
=
\ip{f}{(I^n_{k,l})^*g}_{\tilde G^n_k}
\end{equation}
for all $f\in C^{-\infty}(\tilde G^n_k)$ and $g\in C^\infty(\tilde G^n_l)$.

\begin{lemma}
\label{lma:funk-identities}
The identities~\eqref{eq:funk-identities} hold also when the transforms are defined in the sense of distributions.
\end{lemma}

\begin{proof}
The proofs are straightforward calculations based on the definitions and the identity $(A\circ B)^*=B^*\circ A^*$ for the relevant operators~$A$ and~$B$ acting on smooth functions.
\end{proof}

\subsection{Convolutions}

The group~$SO(n+1)$ acts naturally on~$S^n$ and~$\tilde G^n_k$.
The integral transforms~$F^n_k$ and~$I^n_{k,l}$ and their duals are equivariant under this action.
We will use this action to define convolutions on our function spaces.

For $\eta\in C^\infty(SO(n+1))$ and $f\in C^{-\infty}(S^n)$, we define the convolution $\eta*f\in C^\infty(S^n)$ by
\begin{equation}
\eta*f(x)
=
\int_{SO(n+1)}\eta(g)f(gx)\der g,
\end{equation}
where~$\der g$ is the normalized Haar measure on the group~$SO(n+1)$.
We define the convolution similarly if $f\in C^\infty(\tilde G^n_k)$.

It is a straightforward calculation to observe that
\begin{equation}
\label{eq:conv-equivariance}
\eta*((F^n_k)^*f)
=
(F^n_k)^*(\eta*f)
\end{equation}
for any $\eta\in C^\infty(SO(n+1))$ and $f\in C^\infty(S^n)$.

Alternatively one can define the convolution first for only smooth functions and then extend to the case $f\in C^{-\infty}(S^n)$ by duality.
We denote by~$\eta'$ the reflected version of a function $\eta\in C^\infty(SO(n+1))$, given by $\eta'(g)=\eta(g^{-1})$.
One can define the convolution of a distribution $f\in C^{-\infty}(S^n)$ and a smooth function $\eta\in C^\infty(SO(n+1))$ so that
\begin{equation}
\label{eq:distr-conv}
\ip{\eta*f}{h}_{S^n}
=
\ip{f}{\eta'*h}_{S^n}
\end{equation}
for every $h\in C^\infty(S^n)$.
This agrees with the definition given above.

\subsection{Characterization of the kernel}
\label{sec:funk-proof}

Now we have the tools needed to characterize the kernel of the Funk transform on distributions.
The corresponding result for smooth functions is a classical one:

\begin{lemma}[{\cite[Theorem~1.7, Section~1B, Chapter~III]{book-helgason}}]
\label{lma:funk-smooth}
For any two integers $0<k<n$ the kernel of the Funk transform~$F^n_k$ on smooth functions is precisely the space of odd functions.
\end{lemma}

\begin{theorem}
\label{thm:funk}
For any two integers $0<k<n$ the kernel of the Funk transform~$F^n_k$ on distributions is precisely the space of odd distributions.
\end{theorem}

\begin{proof}
It is easy to observe that any odd distribution is indeed in the kernel.
It therefore suffices to show that an even distribution with vanishing Funk transform has to vanish.

Let~$f$ be an even distribution on~$S^n$ so that $F^n_kf=0$.
Take any smooth functions $\eta\in C^\infty(SO(n+1))$ and $h\in C^\infty(\tilde G^n_k)$.
Combining equations~\eqref{eq:conv-equivariance} and~\eqref{eq:distr-conv} gives
\begin{equation}
\ip{f}{(F^n_k)^*(\eta'*h)}
=
\ip{f}{\eta'*(F^n_k)^*h}
=
\ip{\eta*f}{(F^n_k)^*h}
\end{equation}
and hence
\begin{equation}
F^n_k(\eta*f)
=
\eta*(F^n_kf)
=
0
.
\end{equation}
(That is, the intertwining property~\eqref{eq:conv-equivariance} holds for distributions.)
Since~$\eta*f$ is an even smooth functions, it follows from lemma~\ref{lma:funk-smooth} that $\eta*f=0$.
This holds for any $\eta\in C^\infty(SO(n+1))$.

If we take a sequence $(\eta_m)_{m\in\N}$ of smooth functions on~$SO(n+1)$ converging weakly to the delta function at the identity element, we have $f=\lim_{m\to\infty}(\eta_m*f)$ in the sense of distributions.
Since each~$\eta_m*f$ vanishes, we must have $f=0$.
\end{proof}

Since~$I^n_{k,n-1}$ and~$F^n_k$ take even distributions to even distributions and we have the identity $I^n_{k,n-1}\circ F^n_k=F^n_{n-1}$ by lemma~\ref{lma:funk-identities}, it would have sufficed to show the desired result for $k=n-1$ above.

\begin{remark}
Let us briefly outline an alternative proof of theorem~\ref{thm:funk} (see~\cite{mathoverflow:funk}).
By equation~\eqref{eq:conv-equivariance} the group~$SO(n+1)$ intertwines with the Funk transform.
Therefore the kernel is a closed $SO(n+1)$-invariant subspace.
This implies that $SO(n+1)$-finite functions (functions contained in finite dimensional representations) are dense in it.
All $SO(n+1)$-finite distributions are in fact smooth, and smooth elements in the kernel are known to be odd.
Therefore all distributions in the kernel are limits of odd functions and therefore odd.
For more details, see~\cite{H:groups-and-geometric-analysis}.
\end{remark}

\section{Geodesics and broken rays in spherical symmetry}
\label{sec:geod}

In this section we will first review the basics of geodesics in spherical symmetry and then calculate integrals over functions along these geodesics.

\subsection{Basic facts about geodesics}
\label{sec:geodesic}

If we intersect our manifold $M\subset\R^n$ (defined in the introduction) with any linear subspace of~$\R^n$, we get a totally geodesic submanifold satisfying the Herglotz condition.
Therefore it suffices to study geodesics in the case $n=2$.

In two dimensions it is convenient to use polar coordinates~$(r,\theta)$.
A geodesic is uniquely determined by the location of its tip, the point with smallest~$r$.
Let us denote the coordinates of this tip by~$(r_0,\theta_0)$ and call~$r_0$ the radius of the geodesic.
There are two conserved quantities: the squared speed $c(r(t))^{-2}[r'(t)^2+r(t)^2\theta'(t)^2]=1$ and the angular momentum $c(r(t))^{-2}r(t)^2\theta'(t)=r_0/c(r_0)$.
These two quantities are conserved also when~$c$ has a jump discontinuity --- these conservation laws at a jump point give Snell's law.

These two conserved quantities allow one to carry out most calculations quite simply.
Let us calculate the length of a geodesic with a tip at radius~$r$.
Suppose the geodesic is parametrizesd by arc length as $\gamma\colon[-L,L]\to M$.
Due to symmetry, the length is $2\int_0^L\der t$.
We change the variable of integration from time $t\in(0,L)$ to radius $s\in(r,1)$.
The Herglotz condition ensures that this change of variable is possible, and that $L<\infty$.
Using the two conserved quantities, we find that
\begin{equation}
\begin{split}
\frac{\der s(t)}{\der t}
&=
\sqrt{c(s(t))^2-s(t)^2\theta'(t)^2}
\\&=
\sqrt{c(s(t))^2-r^2c(r)^{-2}c(s(t))^4s(t)^{-2}}.
\end{split}
\end{equation}
After simplification, we find that the length is
\begin{equation}
\label{eq:geod-length}
2L(r)
\coloneqq
2\int_r^1\frac{1}{c(s)}\left(1-\left(\frac{rc(s)}{sc(r)}\right)^2\right)^{-1/2}\der s.
\end{equation}
In fact, the Herglotz condition is equivalent with all maximal geodesics having finite length.

Similarly the angular distance between the endpoints of the geodesic (measured on the universal cover, may exceed~$2\pi$) is
\begin{equation}
\label{eq:geod-angle}
2\alpha(r)
\coloneqq
2\int_r^1\frac{rc(s)}{c(r)s^2}\left(1-\left(\frac{rc(s)}{sc(r)}\right)^2\right)^{-1/2}\der s.
\end{equation}
The integrand differs from that of~\eqref{eq:geod-length} by the term~$\theta'$ which can be found using conservation of angular momentum.

The lengths and angles can be expressed in this simple form regardless of discontinuities in~$c$ as long as the Herglotz condition is satisfied.
The total length or opening angle of a geodesic can be decomposed into parts between the jumps.
For each part where $c\in C^{1,1}$, one can use the above formulas with appropriate changes to the limits of integration.
The total integral is the sum of these parts.
This makes no difference in the formulas since the conserved quantities are conserved across the jumps.

\subsection{Integrals of functions over geodesics}
\label{sec:geod-int}

Suppose for now that $c\in C^{1,1}$ without any jumps.

We consider functions of the form $f(r,\theta)=a(r)e^{ik\theta}$ on~$M$, where $a\in L^2([R,1])$ and $k\in\Z$.
For convenience, we will assume that~$a$ is continuous in the calculations, but the conclusions will hold true for~$L^2$ functions with obvious modifications.
It follows from lemma~\ref{lma:fourier} that any function in~$L^2(A)$ can be written as a sum of such functions, so we lose no generality studying functions of the chosen form.

Consider a geodesic~$\gamma$ parametrized by $[-T,T]\ni t\mapsto(r(t),\theta_0+\omega(t))$ with $r(0)=r_0$ and $\omega(0)=0$.
Using the results of the previous subsection to convert arc length integrals to radius integrals and the identity $\omega(t)=\int_0^t\omega'(\tau)\der\tau$, we have
\begin{equation}
\label{eq:vv1}
\begin{split}
\int_\gamma f\der_g s
&\coloneqq
\int_{-T}^T f(\gamma(t))\der t
\\&=
\int_{-T}^T a(r(t))e^{ik(\theta_0+\omega(t))}\der t
\\&=
e^{ik\theta_0}
\sum_\pm
\int_0^T a(r(t))e^{\pm ik\omega(t)}\der t
\\&=
2e^{ik\theta_0}
\int_0^T a(r(t))\cos(k\omega(t))\der t
\\&=
2e^{ik\theta_0}
\int_{r_0}^1 a(r)
\\&\quad\times
\cos\left(k\int_{r_0}^r \frac{rc(s)^2}{s^2c(r)}H(s;r_0)\der s\right)H(r;r_0)\der r.
\end{split}
\end{equation}
Here
\begin{equation}
\label{eq:H-def}
H(r;z)
=
\frac1{c(r)}\left(1-\left(\frac{zc(r)}{rc(z)}\right)^2\right)^{-1/2}
.
\end{equation}
If we define an integral transform~$\A_k$ taking functions on~$[R,1]$ to functions on~$[R,1]$ by
\begin{equation}
\A_k g(x)
=
2\int_x^1 g(r)\cos\left(k\int_{x}^r \frac{rc(s)^2}{s^2c(r)}H(s;x)\der s\right)H(r;x)\der r,
\end{equation}
we have
\begin{equation}
\label{eq:int-fact}
\int_\gamma f\der_g s
=
e^{ik\theta_0}\A_k a(r_0).
\end{equation}
With this result injectivity of ray transforms can be reduced to injectivity of the integral transform~$\A_k$.

To simplify notation, we define
\begin{equation}
T_k(r;r_0)
=
\cos\left(k\int_{r_0}^r \frac{rc(s)^2}{s^2c(r)}H(s;r_0)\der s\right),
\end{equation}
so that
\begin{equation}
\label{eq:A-def}
\A_k g(x)
=
2\int_x^1 g(r)T_k(r;x)H(r;x)\der r.
\end{equation}
By the analogue to the generalized Abel transforms studied in~\cite{I:disk}, we call~$T_k$ the Chebyshev functions.

To make the integral transform more tractable, we make another change of variable: we change from~$r$ to $\roo r=\rho(r)=r/c(r)$, and denote other variables similarly with tildes.
The Herglotz condition makes~$\rho$ a diffeomorphism.
Multiplicative constants in the function~$c$ are irrelevant, so we may assume that $c(1)=1$ and so $\rho(1)=1$.
We first notice that
\begin{equation}
H(r;z)
=
\frac{\roo r}{c(r)}(\roo r+\roo z)^{-1/2}(\roo r-\roo z)^{-1/2};
\end{equation}
the singularity at $r=z$ is simpler in the new variables.
We denote
\begin{equation}
\roo T_k(\roo r;\roo z)
=
T_k(\rho^{-1}(\roo r);\rho^{-1}(\roo z))
\end{equation}
and
\begin{equation}
\roo H(\roo r;\roo z)
=
\frac{2\roo r}{c(\rho^{-1}(\roo r))\rho'(\rho^{-1}(\roo r))}(\roo r+\roo z)^{-1/2}.
\end{equation}
We write $\roo a=a\circ\rho^{-1}$ (this implies $\roo a(\roo r)=a(r)$) and define the transforms~$\roo\A_k$ so that $\roo\A_k\roo a(\roo x)=\A_k a(x)$.
In this notation
\begin{equation}
\label{eq:abel-cov}
\begin{split}
\roo\A_k\roo a(\roo x)
&=
\int_{\roo x}^{1}\roo a(\roo r) \roo T_k(\roo r;\roo x) \roo H(\roo r;\roo x) (\roo r-\roo x)^{-1/2} \der\roo r.
\end{split}
\end{equation}

\subsection{Attenuated integrals}
\label{sec:att-int}

In section~\ref{sec:geod-int} we calculated the geodesic ray transform of functions of the form $f(r,\theta)=a(r)e^{ik\theta}$ for $k\in\Z$.
Now we will calculate the attenuated geodesic ray transform for a spherically symmetric attenuation~$\lambda$ on~$A$; this calculation will be used to explicitly express the attenuated geodesic ray transform of a general function in~$L^2(A)$ in section~\ref{sec:xrt} below.

Let $\lambda\colon[R,1]\to\R$ be a continuous function, fix some geodesic $\gamma=\gamma_{r_0,\theta_0}$, and define
\begin{equation}
\label{eq:Ipm-def}
I_\pm
=
\int_{-T}^T \Lambda_\pm(t)f(\gamma(t))\der t,
\end{equation}
where
\begin{equation}
\label{eq:Lpm-def}
\Lambda_\pm(t)=\exp\left(\int_{-T}^{\pm t}\lambda(r(\pm s))\der s\right).
\end{equation}
Since $\Lambda_+(0)=\Lambda_-(0)$ and $r(-s)=r(s)$, we have
\begin{equation}
\Lambda_\pm(t)
=
\Lambda_+(0)\exp\left(\pm\int_0^t\lambda(r(s))\der s\right)
\end{equation}
and thus
\begin{equation}
I_++I_-
=
2\Lambda_+(0)\int_{-T}^T \cosh\left(\int_0^t\lambda(r(s))\der s\right) f(\gamma(t))\der t.
\end{equation}
Following the steps of equation~\eqref{eq:vv1}, we obtain
\begin{equation}
\label{eq:att-int-fact}
I_++I_-
=
4E^\lambda(r_0)e^{ik\theta_0}
\int_{r_0}^1 a(r)\Lambda^\lambda(r;r_0)T_k(r;r_0)H(r;r_0)\der r,
\end{equation}
where
\begin{equation}
\label{eq:Lambda-def}
\Lambda^\lambda(r;r_0)
=
\cosh\left(\int_{r_0}^r\lambda(u)H(u;r_0)\der u\right)
\end{equation}
and
\begin{equation}
\label{eq:E-def}
E^\lambda(r_0)
=
\exp\left(\int_{r_0}^{1}\lambda(r(s))H(s;r_0)\der s\right).
\end{equation}

\subsection{Abel-type integral transforms}
\label{sec:abel-integral}

We define the integral transform~$\A_k$ taking functions on~$[\roo R,1]$ to functions on~$[\roo R,1]$ by
\begin{equation}
\label{eq:rA-def}
\roo\A_k g(x)
=
\int_x^1g(r) \roo T_k(r;x) \roo H(r;x) (r-x)^{-1/2} \der r.
\end{equation}
The nonsingular part of the kernel of this integral transform, $K_k(r,x)=\roo T_k(r;x) \roo H(r;x)$ is Lipschitz continuous in~$r$ and~$x$ and $K_k(x,x)>0$ for all~$x$.
(The function~$\roo H$ is obviously Lipschitz, and~$\roo T_k$ can be shown to be, too, with a proof identical to that of lemma~\ref{lma:Lambda}.)
Therefore~$\roo\A_k$ falls in the category of integral transforms studied in section~\ref{sec:abel}.

Analogously with the Abel transform~$\A_k$ defined in equation~\eqref{eq:A-def} we define the attenuated Abel transform
\begin{equation}
\label{eq:A-att-def}
\A_k^\lambda g(x)
=
2\int_x^1 g(r)\Lambda^\lambda(r;x)T_k(r;x)H(r;x)\der r.
\end{equation}
Zero attenuation gives the previously defined Abel transform: $\A_k^0=\A_k$.

The Abel transforms in the following result correspond to any~$C^{1,1}$ wave speed satisfying the Herglotz condition.

\begin{lemma}
\label{lma:abel}
The integral transform~$\roo\A_k$ defined in~\eqref{eq:rA-def} has the following properties:
\begin{enumerate}
\item $\roo\A_k\colon L^p([\roo R,1])\to L^p([\roo R,1])$ is continuous whenever $1/2+1/p<1/q$.
The operator norm of~$\roo\A_k$ is bounded uniformly in~$k$.
\item If~$f$ is continuous or $f\in L^p$ for $p>2$, then~$\roo\A_k f$ is continuous.
\item When acting on functions in~$L^1([\roo R,1])$, the transform~$\roo\A_k$ is injective.
\item If $f\in L^1([\roo R,1])$ and $\roo\A_kf(x)=0$ for all $x\geq z$, then $f(x)=0$ for all $x\geq z$.
\end{enumerate}
Furthermore, the integral transform~$\A_k$ defined in~\eqref{eq:A-def} has the same properties when~$\roo R$ is replaced with~$R$ in the claims.

Moreover, the same is true for the transform~$\A_k^\lambda$ defined in~\eqref{eq:A-att-def} if $\lambda\colon[R,1]\to\R$ is Lipschitz continuous.
\end{lemma}

\begin{proof}
The claims follow from theorem~\ref{thm:abel-cont}, theorem~\ref{thm:If-cont}, and theorem~\ref{thm:inj}.
The transforms~$\A_k$ and~$\roo\A_k$ only differ by the diffeomorphism~$\rho$ (which is bi-Lipschitz), so the claims for~$\A_k$ are equivalent with those for~$\roo\A_k$.
The results for~$\A_k^\lambda$ require additionally lemma~\ref{lma:Lambda} below.
\end{proof}

\begin{lemma}
\label{lma:Lambda}
For Lipschitz continuous $\lambda\colon[R,1]\to\C$ the function~$\Lambda^\lambda$ defined in~\eqref{eq:Lambda-def} is Lipschitz continuous on $\Delta=\{(r,x);R\leq x<r\leq1\}$.
It extends continuously to the diagonal $\{(r,r);R\leq r\leq1\}$ and has the constant value~$1$ on it.
\end{lemma}

\begin{proof}
We define
\begin{equation}
F(r;x)=\int_{x}^r\lambda(u)H(u;x)\der u
\end{equation}
so that $\Lambda^\lambda(r;x)=\cosh(F(r;x))$.
It suffices to show that~$F^2$ is Lipschitz continuous and vanishes on the diagonal, since~$\cosh(F)$ is a Lipschitz function of~$F^2$.

We will again change variables from~$r$ and~$x$ to $\roo r=\rho(r)$ and $\roo x=\rho(x)$.
If we define~$\roo F$ by demanding $\roo F(\roo r;\roo x)=F(r;x)$ and $\roo\lambda=\lambda\circ\rho^{-1}$, we have (cf.~equation~\eqref{eq:abel-cov})
\begin{equation}
\roo F(\roo r;\roo x)
=
\int_{\roo x}^{\roo r}\roo\lambda(\roo u)\roo H(\roo u;\roo x)(\roo u-\roo x)^{-1/2}\der\roo u.
\end{equation}
Since~$\rho$ is bi-Lipschitz, we need to show that~$\roo F^2$ is Lipschitz and vanishes on the diagonal when~$\roo\lambda$ is Lipschitz.

Let~$\lip_1\roo F(\roo r;\roo x)$ denote the local Lipschitz constant of~$\roo F$ with respect to~$\roo r$ at~$(\roo r;\roo x)$, $\lip_2\roo F(\roo r;\roo x)$ similarly, and $\lip\roo\lambda(\roo x)$ the local Lipschitz constant of~$\roo\lambda$ at~$\roo x$.

Elementary estimates yield
\begin{equation}
\label{eq:l1F-est}
\lip_1\roo F(\roo r;\roo x)
\leq
C(\roo r-\roo x)^{-1/2}
\end{equation}
and
\begin{equation}
\label{eq:abs-F-est}
\abs{\roo F(\roo r;\roo x)}
\leq
C(\roo r-\roo x)^{1/2}
\end{equation}
for some constant~$C$.

For fixed~$\roo r$, the function $\roo F(\roo r;\cdot)$ is an Abel-type integral transform of~$\roo\lambda$ in the sense of section~\ref{sec:abel}.
Therefore we have from theorem~\ref{thm:If-lip} that
\begin{equation}
\label{eq:l2F-est}
\lip_2\roo F(\roo r;\roo x)
\leq
C(\roo r-\roo x)^{-1/2}
\end{equation}
for some constant~$C$.

Combining estimates~\eqref{eq:l1F-est}, \eqref{eq:abs-F-est} and~\eqref{eq:l2F-est} we get
\begin{equation}
\lip_i \roo F^2(\roo r;\roo x)
\leq
2\abs{\roo F(\roo r;\roo x)}\lip_i \roo F(\roo r;\roo x)
\leq
2C^2
\end{equation}
for $i=1,2$, and~$\roo F^2$ is Lipschitz.
Estimate~\eqref{eq:abs-F-est} shows that~$\roo F^2$ vanishes on the diagonal, and the proof is complete.
\end{proof}

\begin{remark}
\label{rmk:abel}
Proposition~\ref{prop:fact-inj} gives a simple inversion formula for some Abel transforms.
The transforms~$\roo\A_0$ and~$\A_0$ do not fall in this category, but a calculation verifies an inversion formula for~$\A_0$.
We do not know of such a simple formula for other~$k$, except in the Euclidean case (see~\cite[lemma~10(5)]{I:disk}).
The inversion formula for~$\A_0$ is
\begin{equation}
\begin{split}
f(r)
&=
-\frac{c(r)}{\pi}
\left.\Der{x}\int_{x}^{1}\frac{\rho'(x)}{\rho(x)}\left(\left(\frac{\rho(z)}{\rho(x)}\right)^2-1\right)^{-1/2}\A_0f(z)\der z\right|_{x=r}
,
\end{split}
\end{equation}
and it holds almost everywhere for $f\in L^1([R,1])$ and everywhere for continuous~$f$.
\end{remark}

\subsection{Broken rays}
\label{sec:br}

Due to spherical symmetry, studying broken rays is simple once one understands geodesics.
Just like a non-radial geodesic, a non-radial broken ray is contained in a unique two-dimensional subspace of~$\R^n$ and it therefore suffices to consider dimension two.

The geodesic segments that constitute a broken ray only differ by rotations of the plane.
Each of them has the same radius (radial coordinate of the tip), so we may call it the radius of the broken ray.
It follows from spherical symmetry that if two distinct rotations of a geodesic share an endpoint, then they together form a broken ray that satisfies the usual reflection condition: the angle of reflection equals the angle of incidence.

Not all broken rays are periodic.
Due to rotation symmetry periodicity only depends on the radius.
Using the angle~$\alpha$ defined in~\eqref{eq:geod-angle}, it is easy to see that the broken ray corresponding to $r\in(R,1)$ is periodic if and only if $\alpha(r)\in\pi\Q$.

To prove injectivity results for the periodic broken ray transform we need to have enough periodic broken rays.
One such result is provided by the following proposition, which assumes the countable conjugacy condition.
This proposition is the sole reason for making the assumption in the results presented in this paper.

\begin{proposition}
\label{prop:ccc-dense}
If a~$C^{1,1}$ wave speed satisfies the Herglotz condition (definition~\ref{def:herglotz}) and the countable conjugacy condition (definition~\ref{def:ccc}), then the set of radii corresponding to periodic broken rays is countable and dense in~$(R,1)$.
\end{proposition}

The proof of the proposition is somewhat involved, so we will not present it here.
We only mention that by proposition~\ref{prop:abel-diff} the opening angle~$\alpha(r)$ defined in~\eqref{eq:geod-angle} is~$C^1$ and that conjugate points at the boundary correspond to zeros of the derivative~$\alpha'$.
A proof can be found in~\cite[Lemma~4.5]{HIK:spherical-spectral}.

\section{The X-ray transform}
\label{sec:xrt}

Let us recall the definitions of the X-ray transform and the attenuated X-ray transform on a manifold~$M$ with boundary.
The X-ray transform of a function $f\colon M\to\C$ is a function on the space of maximal geodesics.
Evaluated at a maximal geodesic $\gamma\colon[0,T]\to M$ it is
\begin{equation}
\rt f(\gamma)
=
\int_0^Tf(\gamma(t))\der t.
\end{equation}
The linear operator~$\rt$ is the X-ray transform.
Here the midpoint of the geodesic is~$\gamma(T/2)$, not~$\gamma(0)$.

Take any continuous function $\lambda\colon M\to\R$.
The attenuated X-ray transform with attenuation~$\lambda$ is defined via
\begin{equation}
\rt^\lambda f(\gamma)
=
\int_0^T\exp\left(\int_0^t\lambda(\gamma(s))\der s\right)f(\gamma(t))\der t.
\end{equation}
Notice that if the attenuation vanishes, then the attenuated X-ray transform is just the X-ray transform: $\rt^0=\rt$.


\begin{theorem}
\label{thm:xrt}
Let $M=\bar B(0,1)\setminus\bar B(0,R)$, $R\in(0,1)$ and $n\geq2$, with the metric $g(x)=c^{-2}(\abs{x})e(x)$ as in the introduction.
Suppose the radial wave speed~$c$ is piecewise~$C^{1,1}$ and satisfies the Herglotz condition.
Let~$\lambda$ be a radially symmetric Lipschitz continuous function $M\to\R$.
Then the attenuated X-ray transform~$\rt^\lambda$ is injective on~$L^2(M)$.

In particular, any function $f\in L^2(M)$ is uniquely determined by its integrals over all geodesics.
\end{theorem}

\begin{remark}
The theorem is stated in a setting where a small ball of (coordinate) radius~$R$ is removed from a spherically symmetric manifold.
The theorem is also applicable in a whole ball without anything removed.
One can artificially remove a small ball of radius~$R$ from the center of the manifold and use the theorem for all $R>0$ to conclude that the attenuated X-ray transform is injective.
In fact, it suffices that all regularity assumptions (on~$f$, $\lambda$ and~$c$) are satisfied locally in the punctured ball --- the functions may blow up at the origin.
\end{remark}

\begin{remark}
Theorem~\ref{thm:xrt} can be seen as a support theorem.
It follows immediately from the theorem that if a function on the spherically symmetric manifold integrates to zero over all geodesics that stay outside a ball centered at the (coordinate) origin, then the function must vanish outside that ball.
In Euclidean geometry --- which is a special case of the theorem --- this is Helgason's famous support theorem~\cite[Theorem~2.6]{book-helgason}.
\end{remark}

\begin{remark}
Theorem~\ref{thm:xrt} provides injectivity, but not stability.
It is well known that a spherically symmetric metric satisfying the Herglotz condition can have conjugate points.
In dimension two the existence of conjugate points immediately implies instability for the unweighted X-ray transform~\cite{MSU:x-ray-cp}.
\end{remark}

\begin{proof}[Proof of theorem~\ref{thm:xrt}]
Fix a Lipschitz attenuation $\lambda\colon[R,1]\to\R$.
Suppose $f\in L^2(M)$ satisfies $\rt^\lambda f=0$.
We wish to show that $f=0$; this implies injectivity.

If we intersect our manifold~$M$ with a two dimensional subspace of~$\R^n$, we get a totally geodesic submanifold with radial symmetry.
The restriction of a function $f\in L^2(M)$ is still~$L^2$ for almost all of these submanifolds by lemma~\ref{lma:lp-restriction}.
If the theorem holds in dimension two, then~$f$ vanishes on almost all of these submanifolds and therefore on almost all of~$M$.
Therefore it suffices to prove the theorem in dimension two.

Suppose the result is true for~$C^{1,1}$ wave speeds.
Let $R<a_1<a_2<\dots<a_N<1$ be the points where~$c$ fails to be~$C^{1,1}$.
Consider the submanifold $M'\subset M$ where the radial coordinate is restricted to~$(a_N,1]$.
On~$M'$ the metric is~$C^{1,1}$ and satisfies the assumptions.
Applying the result on the manifold~$M'$ shows that $f|_{M'}=0$.

Then we may restrict our problem to $M\setminus M'$, since~$f$ vanishes elsewhere.
In the same way we can deduce that~$f$ must vanish also between radii~$a_{N-1}$ and~$a_N$.
Continuing this layer stripping argument at discontinuities of~$c$ eventually shows that~$f$ must vanish in all of~$M$.

Therefore it suffices to prove the theorem when $c\in C^{1,1}$.
It is again convenient to use polar coordinates.
We may write~$f(r,\theta)$ as a Fourier series in~$\theta$.
By lemma~\ref{lma:fourier} there is a sequence of functions $a_k\in L^2(R,1)$ so that
\begin{equation}
f(r,\theta)
=
\sum_{k\in\Z}a_k(r)e^{ik\theta},
\end{equation}
where the series converges in~$L^2(M)$.
Let us denote $f_k(r,\theta)=a_k(r)e^{ik\theta}$.

For any $(r,\theta)\in M$ there are two geodesics with their tip at this point.
They are reverses of each other.
If these geodesics are denoted by~$\gamma_\pm(r,\theta)$, we write
\begin{equation}
\rt^\lambda_0f(r,\theta)
=
\rt^\lambda f(\gamma_+(r,\theta))
+
\rt^\lambda f(\gamma_-(r,\theta)).
\end{equation}
Combining equations~\eqref{eq:att-int-fact} and~\eqref{eq:A-att-def}, we have
\begin{equation}
\rt^\lambda_0f_k(r,\theta)
=
e^{ik\theta}2E^\lambda(r)\A^\lambda_ka_k(r).
\end{equation}
See section~\ref{sec:att-int} for details.

It is easy to check that~$E^\lambda(r)$ is bounded, and lemma~\ref{lma:abel} guarantees that the integral transforms $A^\lambda_k\colon L^2(R,1)\to L^2(R,1)$ are equicontinuous.
Therefore $\rt^\lambda_0\colon L^2(M)\to L^2(M)$ is continuous and
\begin{equation}
\rt^\lambda_0f(r,\theta)
=
2E^\lambda(r)\sum_{k\in\Z}e^{ik\theta}\A^\lambda_ka_k(r).
\end{equation}
Since~$E^\lambda(r)$ never vanishes and we assumed that $\rt^\lambda f=0$, we have
$
\A^\lambda_ka_k=0
$
for every $k\in\Z$.
The Abel-type integral transform~$\A^\lambda_k$ is injective by lemma~\ref{lma:abel}, so in fact $a_k=0$ for all~$k$.
This means that~$f$ indeed vanishes identically.
\end{proof}

\section{The broken ray transform}
\label{sec:brt}

Injectivity results for the broken ray transform in the Euclidean disc or ball were shown in~\cite{I:disk}.
In this section we generalize those results to radially symmetric manifolds satisfying the Herglotz condition and the countable conjugacy condition.
The Euclidean metric satisfies both conditions.

We will only give abridged versions of the proofs, focusing on the new ingredients due to a non-Euclidean metric.
The missing details can be found in the Euclidean arguments presented in~\cite{I:disk}.

By the boundary~$\partial M$ of our spherically symmetric manifold we only mean the outer boundary where $r=1$.
Let us choose a subset $E\subset\partial M$.
The set~$E$ is the set of tomography, where broken rays have their initial and final points, and the set $\partial M\setminus E$ is where the broken rays reflect.
All measurements are therefore done on the set of tomography, and the rest of the boundary is essentially a mirror.

The broken ray transform takes a function on~$M$ into a function on the set of broken rays.
Injectivity of this transform depends on the set of tomography~$E$.
The bigger~$E$ is, the more data are available.

\begin{theorem}
\label{thm:brt}
Let $R\in[0,1)$ and $n\geq2$.
Let $M=\bar B(0,1)\setminus\bar B(0,R)\subset\R^n$ be equipped with a metric corresponding to a radially symmetric~$C^{1,1}$ wave speed that satisfies the Herglotz condition (definition~\ref{def:herglotz}) and the countable conjugacy condition (definition~\ref{def:ccc}) condition.
\begin{enumerate}
\item If $f\in C(M)$ is continuous and the set of tomography is a singleton, then integrals of~$f$ over broken rays uniquely determine the integral of~$f$ over any circle centered at the origin with the singleton in the circle's plane.
\item Suppose $f\colon M\to\C$ is uniformly quasianalytic in the angular variables in the sense of~\cite[Definition~14]{I:disk} and all its angular derivatives satisfy the Dini--Lipschitz condition.
If the set of tomography is open, the function~$f$ is uniquely determined by its integrals over all broken rays.
\end{enumerate}
\end{theorem}

\begin{remark}
In the first case, if $R=0$ and~$f$ is continuous in the origin, then the broken ray transform of~$f$ determines~$f(0)$.
However, the broken ray transform does not determine the value of the function at any other point if the set of tomography is a singleton.
\end{remark}

\begin{remark}
It is an open problem whether the broken ray transform is injective on smooth functions in the Euclidean unit disc.
\end{remark}

\begin{proof}[Proof of theorem~\ref{thm:brt}]
It suffices to prove the theorem with $R>0$; the results for $R=0$ follow from combining the results for all positive~$R$.
It is enough to prove both statements in dimension two.

(1)
All broken rays are in fact periodic broken rays, since the initial and final points must coincide.
The set of radii corresponding to such broken rays is dense by proposition~\ref{prop:ccc-dense}.
In fact, for any $N\in\N$, the set of radii of broken rays with at least~$N$ reflections is dense.
This can be seen by induction, as excluding one number of reflections does not affect the density in the absence of conjugate points.

Fix any $z\in(R,1)$.
There is a sequence of broken rays~$\gamma_i$ so that the number of reflections on~$\gamma_i$ tends to infinity as~$i$ increases,~$\gamma_i$ meets~$E$ only at its endpoints and the smallest radial coordinate on~$\gamma_i$ tends to~$z$.
Using the aforementioned density result, one can find a sequence of periodic broken rays whose radius converges to~$z$ and whose number of reflections increases without bound.

The average integral of~$f$ over~$\gamma_i$ tends to
\begin{equation}
\frac{\A_0a_0(z)}{\A_01(z)},
\end{equation}
where~$\A_0$ is the Abel-type transform defined in~\eqref{eq:A-def}, $1$ denotes the constant function, and~$a_0$ is the radially symmetrized~$f$ (as in lemma~\ref{lma:fourier}).
In essence, the idea is that in the limit $i\to\infty$ the normalized integral over the broken ray~$\gamma_i$ tends to a rotation symmetric measure corresponding to~$\A_0$.
For details of this calculation, see~\cite[Section~2.3]{I:disk}.

It is easy to check that $\A_01(z)>0$ for $z<1$, so the data determine the function~$\A_0a_0$.
The transform~$\A_0$ is injective by lemma~\ref{lma:abel}, so the data determine~$a_0$.
This is exactly the claim.

(2)
Suppose~$f$ integrates to zero over all broken rays.
We may write it as a Fourier series as in lemma~\ref{lma:fourier}, and the series converges uniformly (cf.~\cite[Lemma~8]{I:disk}).
By the first part of the theorem we already know that $a_0=0$.

Since~$f$ integrates to zero over all broken rays, so do its angular derivatives of all orders.
We shall use angular derivatives of even orders,~$\partial_{\theta}^{2n}f$.

We may choose our polar coordinates~$(r,\theta)$ so that the point~$(1,0)$ is in~$E$.
Consider a broken ray~$\gamma$ whose initial point is at $\theta=-\omega$ and final point at $\theta=\omega$ for some small $\omega>0$.
Let~$z_\gamma$ denote the distance from~$\gamma$ to the origin in~$\R^2$.
Then the average of~$\partial_{\theta}^{2n}f$ over the broken ray~$\gamma$ is
\begin{equation}
\label{eq:vv2}
\sum_{k\in\Z}(-k^2)^nS_k(\gamma)\A_ka_k(z_\gamma).
\end{equation}
Here~$S_k(\gamma)$ is a coefficient depending on the broken ray defined in~\cite[Section~4.1]{I:disk}.
It has exactly the same form in the non-Euclidean setting.

To find formula~\eqref{eq:vv2}, one first needs to recall that if~$f$ is written as an angular Fourier series
\begin{equation}
f(r,\theta)
=
\sum_{k\in\Z}e^{ik\theta}a_k(r),
\end{equation}
then the integral of~$f$ over a geodesic with a tip at $(r,\theta)$ is
\begin{equation}
\rt f(r,\theta)
=
\sum_{k\in\Z}e^{ik\theta}\A_ka_k(r).
\end{equation}
The Euclidean version of this statement was given (implicitly) in~\cite[eq.~(37)]{I:disk} and the more general statement in radial symmetry follows form~\eqref{eq:int-fact}.
Suppose then that the corresponding geodesic has opening angle $\alpha(r)\in\pi\Q$ and the corresponding periodic broken ray has~$N$ reflections.
The integral over the boken ray is then
\begin{equation}
\sum_{l=1}^N\rt f(r,\theta+l\alpha).
\end{equation}
This sum over~$l$ can be done explicitly for each term~$e^{ik\theta}$, and this leads to the coefficients~$S_k(\gamma)$.
This argument is the same in Euclidean geometry; the only difference is in the definition of the Abel transforms~$\A_k$.

Since the sum~\eqref{eq:vv2} vanishes for all $n\in\N$, then $S_k(\gamma)\A_ka_k(z_\gamma)=0$ for all $k\in\Z$.
The assumption of quasianalyticity is needed for this step.
The coefficient~$S_k(\gamma)$ is non-zero whenever the broken ray~$\gamma$ cannot be extended to a periodic one --- or equivalently, whenever $\omega\notin\pi\Q$.
The set of corresponding~$z_\gamma$ is dense in~$(R,1)$ by proposition~\ref{prop:ccc-dense} (the complement is countable).
Thus each of the functions~$\A_ka_k$, $k\in\Z$, vanishes in a dense set.

We then turn to lemma~\ref{lma:abel} for properties of Abel transforms.
The functions~$a_k$ are continuous, and so are~$\A_ka_k$.
The transforms are injective, so all of the functions~$a_k$ must vanish.
This concludes the proof.
\end{proof}

\section{The periodic broken ray transform}
\label{sec:pbrt}

In this section we consider the periodic broken ray transform.
The related inverse problem is to reconstruct a function on a manifold with boundary from its integrals over all periodic broken rays.

It turns out that in spherical symmetry the transform does not contain enough information to recover the entire function.
In dimension three or higher one can recover the even part but very little information about the odd part.
In dimension two one can recover the circlewise average (like in theorem~\ref{thm:brt}) but very little other information.

The periodic broken ray transform was previously known to be injective in the Euclidean square but non-injective in the Euclidean disc~\cite{I:refl}.

\begin{theorem}
\label{thm:pbrt}
Let $R\in[0,1)$.
Let $M=\bar B(0,1)\setminus\bar B(0,R)\subset\R^n$ be equipped with a metric corresponding to a radially symmetric~$C^{1,1}$ wave speed that satisfies the Herglotz condition (definition~\ref{def:herglotz}) and the countable conjugacy condition (definition~\ref{def:ccc}) condition.
\begin{enumerate}
\item
\label{item:2yes}
For $n=2$:
Integrals over all periodic broken rays of a function in~$L^p(M)$, $p>3$, uniquely determine the spherical average (zeroth Fourier component with respect to the angular variable) of the function.
\item
\label{item:3yes}
For $n\geq3$:
Integrals over all periodic broken rays of a function in~$L^p(M)$, $p>3$, uniquely determine the symmetric part of the function.
\item
\label{item:2no}
For $n=2$:
Assume that the wave speed satisfies the finite conjugacy condition (definition~\ref{def:fcc}).
Let $Y\colon S^1\to\C$ be a trigonometric polynomial with zero average.
The space of such $a\in C_0^\infty(R,1)$ that the function $M\ni r\theta\mapsto a(r)Y(\theta)$ integrates to zero over all periodic billiard trajectories has finite codimension in~$C_0^\infty(R,1)$.
In particular, the periodic broken ray transform has an infinite dimensional kernel in~$C_0^\infty(M)$.
\item
\label{item:3no}
For $n\geq2$:
Assume that the wave speed satisfies the finite conjugacy condition (definition~\ref{def:fcc}).
If $Y\colon S^{n-1}\to\C$ is a finite sum of odd spherical harmonics, the space of such functions $a\in C_0^\infty(R,1)$ that the function $M\ni r\omega\mapsto Y(\omega)a(r)$ integrates to zero over all periodic broken rays has finite codimension in~$C_0^\infty(R,1)$.
In particular, the kernel of the periodic broken ray transform is infinite dimensional (but contained in the space of odd functions if $n\geq3$).
\end{enumerate}
\end{theorem}

\begin{remark}
We may without loss of generality assume that $R>0$ in the proof, and we do so without mention.
This choice eliminates radial geodesics.
Letting $R\to0$ easily gives the result for $R=0$ once it has been proven for all $R>0$.
\end{remark}

\begin{remark}
A version of this problem is considered in~\cite{HIK:spherical-spectral} in connection with spectral rigidity problems.
For rigidity within a radially symmetric class, one only needs to consider radially symmetric functions, and for those we do indeed have injectivity for $n\geq2$.
\end{remark}

The tools needed for this proof can be found in the subsections below.
Proofs of different parts of the theorem are given in sections~\ref{sec:pbrt3even} and~\ref{sec:pbrt2odd}.
Each part of the theorem is given a separate proof.

\subsection{The planar average ray transform}

Let~$G^n_2$ be the Grassmannian of two-dimensional subspaces of~$\R^n$.
Notice that every non-radial geodesic is contained in exactly one plane in~$G^n_2$ and that $M\cap P$ is a totally geodesic submanifold of~$M$ for any $P\in G^n_2$.
These submanifolds inherit properties from~$M$: the Herglotz condition, the countable conjugacy condition, absence of conjugate points, regularity, and also the negations of these conditions.

Let us now define the planar average ray transform~$\prt$, an integral transform closely related to the periodic broken ray transform.
The planar average ray transform~$\prt f$ of a sufficiently regular function $f\colon M\to\C$ is a function on $G^n_2\times(R,1)$.
If~$If(\gamma)$ denotes the integral of~$f$ over a geodesic~$\gamma$, then~$\prt f(P,r)$ is the average of~$If(\gamma)$ over all geodesics~$\gamma$ of radius~$r$ and contained in the plane~$P$.
If $n=2$, the set~$G^n_2$ is a singleton and we consider~$\prt f$ to be a function of radius only.

In dimension two, let~$f_0$ denote the angular average of a function $f\colon M\to\C$, that is, $f_0(r)=\fint_{\partial B(0,r)}f(x)\h^1(x)$.

In the lemmas below,~$\A_0$ is the Abel transform corresponding to the conformal factor~$c(r)$ defined by equation~\eqref{eq:A-def}.

\begin{lemma}
\label{lma:part-cont}
Assume the assumptions of theorem~\ref{thm:pbrt} and let $n=2$.
If $f\in C(M)$, then $\prt f(r)=\A_0f_0(r)$ for all $r\in(R,1)$.
\end{lemma}

\begin{proof}
Fix any $r\in(R,1)$.
We have
\begin{equation}
\int_{\gamma_{r,\theta}} f\der_g s
=
\sum_\pm
\int_{r_0}^1 f\left(r,\theta\pm\int_{r_0}^r \frac{rc(s)^2}{s^2c(r)}H(s;r_0)\der s\right)H(r;r_0)\der r
\end{equation}
for any angle $\theta\in S^1$; cf. equation~\eqref{eq:vv1}.
If we take the integral average over $\theta\in S^1$, the left-hand side becomes~$\prt f(r)$ and the right-hand side becomes~$\A_0f_0(r)$.
\end{proof}

\begin{lemma}
\label{lma:main-part}
Assume the assumptions of theorem~\ref{thm:pbrt}.
Let $n=2$ and fix an exponent $p\in[1,\infty)$.
For almost every $r\in(R,1)$ we have $\prt f(r)=\A_0 f_0(r)$.
If $p>2$, this holds for every~$r$.
\end{lemma}

\begin{proof}
We wish to show that for almost every $r\in(R,1)$ (every~$r$ if $p>2$) our function $f\in L^p(M)$ is integrable over almost every geodesic of radius~$r$, and that $\prt f=\A_0f_0$.

Let us define the operator $F\colon L^p(M)\to L^p(R,1)$ by $Ff=\A_0f_0$.
The mapping $L^p(M)\ni f\mapsto f_0\in L^p(R,1)$ is clearly continuous, and so is $\A_0\colon L^p(R,1)\to L^p(R,1)$ (by lemma~\ref{lma:abel}).

Since continuous functions are dense, we can take a sequence of functions $f_k\in C(M)$ tending to~$f$ in~$L^p(M)$.
Since $F\abs{f_k}\to F\abs{f}$ in~$L^p(R,1)$, we can --- after passing to a subsequence --- assume that $F\abs{f_k}(r)\to F\abs{f}(r)$ and $f_k(r)\to f(r)$ for almost every $r\in(R,1)$.
If $p>2$, we have $F\abs{f_k}(r)\to F\abs{f}(r)$ for every~$r$.

Take any such~$r$.
By lemma~\ref{lma:part-cont} we have the bound
\begin{equation}
\begin{split}
\abs{\fint_{S_1}\int_{\gamma_{r,\theta}}f_k\der s\der\theta}
&\leq
\fint_{S_1}\int_{\gamma_{r,\theta}}\abs{f_k}\der s\der\theta
\\&=
F\abs{f_k}(r).
\end{split}
\end{equation}
Due to our choice of~$r$, this is a bounded sequence in~$k$.
Since $f_k\to f$ almost everywhere, we have by the dominated convergence theorem that
\begin{equation}
\fint_{S_1}\int_{\gamma_{r,\theta}}f\der s\der\theta
\end{equation}
exists and is finite.
This implies that~$\prt f(r)$ makes sense and actually $\prt f(r)=\A_0f_0(r)$.
\end{proof}

In the planar average ray transform we may restrict the set of radii to a subset of~$(R,1)$, as we will do next.
The following theorem characterizes the kernel of the planar average ray transform in dimensions $n\geq3$ as the set of odd functions.

\begin{theorem}
\label{thm:part}
Let $R\in[0,1)$.
Let $M=\bar B(0,1)\setminus\bar B(0,R)\subset\R^n$, $n\geq3$, be equipped with a metric corresponding to a radially symmetric~$C^{1,1}$ wave speed that satisfies the Herglotz condition (definition~\ref{def:herglotz}).
Fix a set $E\subset(R,1)$ and consider the planar average ray transform~$P$ restricted to radii in~$E$.
\begin{enumerate}
\item If~$E$ has full measure, then~$\prt$ is well defined and determines the even part on~$L^p$ for any $p\geq1$.
\item If~$E$ is dense, then~$\prt$ is well defined and determines the even part on~$L^p$ for any $p>2$.
\item The planar average ray transform~$\prt$ vanishes on all odd functions.
\end{enumerate}
That is, the kernel of~$\prt$ is precisely the set of odd functions (under either of the assumptions on~$E$ and~$p$ given above).
\end{theorem}

\begin{proof}
The last conclusion follows immediately from symmetry considerations, so we focus on the first two.

Let $f\in L^p(M)$ and suppose $\prt f(P,r)=0$ for every $r\in E$ and almost every $P\in G^n_2$.
We wish to show that~$f$ has to be odd.

For almost every plane $P\in G^n_2$ the function $f^P=f|_{P\cap M}$ belongs to $L^p(P\cap M)$ due to lemma~\ref{lma:lp-restriction}.
The planar average ray transform of~$f^P$ can be considered a function of radius only.
By lemma~\ref{lma:main-part} we have $\A_0 f^P_0(r)=\prt f^P(r)=\prt f(P,r)=0$ for all $r\in E$.

We have $f^P_0\in L^p(R,1)$.
Let us consider the two remaining cases of the theorem:
\begin{enumerate}
\item If $p\geq1$, we know that $\A_0 f^P_0\in L^1(R,1)$ (see lemma~\ref{lma:abel}). Since~$E$ has full measure, we must have $\A_0 f^P_0=0$ (as an elment of~$L^1$).
\item If $p>2$, we know that $\A_0 f^P_0\in C(R,1)$ (see lemma~\ref{lma:abel}). Since~$E$ is dense, we have $\A_0 f^P_0=0$.
\end{enumerate}
Either way, the function~$\A_0 f^P_0$ vanishes.

Now injectivity of the Abel transform --- guaranteed by lemma~\ref{lma:abel} --- implies that $f^P_0(r)=0$ for almost every $r\in(R,1)$.
This means that for almost every $r\in(R,1)$ the function~$f|_{\partial B(0,r)}$ is in the kernel of the Funk transform, identifying spheres of different radii by scaling.
The function~$f$ is in~$L^1$ for almost every sphere, so by theorem~\ref{thm:funk} (with $k=1$) the function~$f$ must be odd on almost every sphere.
This concludes the proof.
\end{proof}

\subsection{Positive results}
\label{sec:pbrt3even}

We are now ready to start proving theorem~\ref{thm:pbrt}.
We give the positive results here and the negative ones in the next subsection.

\begin{proof}[Proof of theorem~\ref{thm:pbrt}, part~\ref{item:2yes}]
Consider $f\in L^p(M)$ with vanishing periodic broken ray transform.

Let $P\subset(R,1)$ be the set of radii corresponding to periodic broken rays.
By averaging over rotations, we observe that the periodic broken ray transform determines the planar average ray transform with radii restricted to the set~$P$.
The set~$P$ is dense by proposition~\ref{prop:ccc-dense}.

It follows from lemma~\ref{lma:main-part} that $\A_0f_0(r)=\prt f(r)$ for every $r\in(R,1)$.
Since~$\prt f$ vanishes in the dense set~$P$ and~$\A_0f_0$ is continuous by lemma~\ref{lma:abel}, we know that $\A_0f_0(r)=0$ for every $r\in(R,1)$.
It follows from injectivity of~$\A_0$ (due to lemma~\ref{lma:abel}) that $f_0=0$.
\end{proof}

\begin{proof}[Proof of theorem~\ref{thm:pbrt}, part~\ref{item:3yes}]
Suppose the periodic broken ray transform of $f\in L^p(M)$ is known.
We begin as in the previous proof and observe that the planar average ray transform of~$f$ is thus determined for radii in~$P$.
Theorem~\ref{thm:part} allows us to reconstruct the even part of~$f$ from this information since~$P$ is dense.
\end{proof}

\subsection{Negative results}
\label{sec:pbrt2odd}

It now remains to prove the non-injectivity results contained in theorem~\ref{thm:pbrt}.
In both claims we additionally assume the finite conjugacy condition.
It follows from this condition that preimages of singletons under the function~$\alpha$ defined by~\eqref{eq:geod-angle} are finite.
To see this, observe that by the mean value theorem~$\alpha'$ has to vanish between any two points with equal value of~$\alpha$, but the assumed condition implies that~$\alpha'$ has finitely many zeros.
A similar statement is true in the countable case, see~\cite[Lemma~4.5]{HIK:spherical-spectral}.

\begin{proof}[Proof of theorem~\ref{thm:pbrt}, part~\ref{item:2no}]
The angle~$\alpha(r)$ was discussed in section~\ref{sec:geodesic}, and it was observed in section~\ref{sec:br} that a trajectory of radius~$r$ is periodic if and only if $\alpha(r)\in\pi\Q$.
Let us now calculate the integrals of functions over these trajectories one Fourier component at a time.

Suppose $f(r,\theta)=a(r)e^{ik\theta}$ for some continuous function~$a$.
Due to equation~\eqref{eq:int-fact} the integral of~$f$ over the geodesic parametrized by~$(r,\theta)$ is~$e^{ik\theta}\A_ka(r)$.

We parametrize periodic trajectories like the geodesics that constitute them.
This parametrization is redundant --- each trajectory is described as many times as it has geodesic segments --- but it does not matter.
We say that a radius $r\in(R,1)$ has index $m\in\N$ if the corresponding billiard trajectory is periodic and has~$m$ reflections.
The index~$m$ of a radius~$r$ is the smallest natural number so that $m\alpha(r)\in\pi\N$.

Let~$r$ have index~$m$ and take any $\theta\in S^1$.
Then the integral of~$f$ over the trajectory given by~$(r,\theta)$ is
\begin{equation}
\brt f(r,\theta)
\coloneqq
\sum_{l=0}^{m-1}e^{ik(\theta+2l\alpha(r))}\A_ka(r)
=
e^{ik\theta}\A_ka(r)\sum_{l=0}^{m-1}e^{2ikl\alpha(r)}.
\end{equation}
Here~$\A_k$ is the Abel-type integral transform introduced in sections~\ref{sec:geod-int} and~\ref{sec:abel-integral}.
If $k\alpha(r)\in\pi\N$, then the sum is simply~$m$.
Otherwise it is zero, since $e^{2ikm\alpha(r)}=1$.

The condition $k\alpha(r)\in\pi\N$ is equivalent with~$m$ dividing~$k$, so the integral is
\begin{equation}
\brt f(r,\theta)
=
m\delta_{m\mid k}e^{ik\theta}\A_ka(r),
\end{equation}
where we have denoted
\begin{equation}
\delta_{m\mid k}
=
\begin{cases}
1 & \text{if } m \text{ divides } k\\
0 & \text{otherwise}.
\end{cases}
\end{equation}
Notice that if~$k$ is kept fixed and non-zero, this integral can only be nonzero for a finite amount of indices~$m$.
By standard convention every integer divides zero.

For an integer $m\geq1$, let us denote by $Q_m\subset(R,1)$ the set of radii of index~$m$.
Each set~$Q_m$ is finite, but the union~$\bigcup_{m\in\N}Q_m$ is dense by proposition~\ref{prop:ccc-dense}.

By assumption
\begin{equation}
Y(\theta)
=
\sum_{k=-K}^K y_ke^{ik\theta}
\end{equation}
for some natural number~$K$ and some coefficients~$y_k$.
The assumption of zero average translates to $y_0=0$.

For $a\in C_0^\infty(R,1)$, let us denote $f_a(r,\theta)=a(r)Y(\theta)$.
For a radius~$r$ of index~$m$ we have by the above considerations
\begin{equation}
\brt f_a(r,\theta)
=
\sum_{k=-K}^K y_ke^{ik\theta} m\delta_{m\mid k}\A_ka(r).
\end{equation}
In particular, $\brt f_a(r,\theta)$ can only be nonzero if
\begin{equation}
r\in Q^K\coloneqq\bigcup_{k=1}^KQ_k.
\end{equation}
Therefore the space
\begin{equation}
\begin{split}
E
&=
\{a\in C_0^\infty(R,1); \A_ka(r)=0
\\&\qquad
\text{for all }1\leq k\leq K\text{ and }r\in Q^K\}
\end{split}
\end{equation}
is a subspace of the space we set out to study and it suffices to show that~$E$ has finite codimension in~$C_0^\infty(R,1)$.
But since~$Q^K$ is finite, the elements of~$E$ only satisfy a finite number of linear conditions.
Therefore the codimension of~$E$ is indeed finite.
\end{proof}

\begin{proof}[Proof of theorem~\ref{thm:pbrt}, part~\ref{item:3no}]
This follows from part~\ref{item:2no}.
We only need to know that when an odd spherical harmonic is restricted to a two dimensional subspace (which after intersecting with the sphere means a great circle) is an odd trigonometric polynomial with order bounded by that of the spherical harmonic.
To see this, one can write the spherical harmonic as a harmonic polynomial of the same order as the spherical harmonic; its restriction to a subspace is a polynomial of at most the same order, and the order of a polynomial as a spherical harmonic is bounded by the order of the polynomial.
\end{proof}

\subsection*{Acknowledgements}

M.V.dH\ gratefully acknowledges support from the Simons Foundation under the MATH + X program, and the National Science Foundation under grant DMS-1559587.
J.I.\ was partly supported by an ERC starting grant (grant agreement no 307023) and by the Academy of Finland (decision 295853).
The second author is grateful for hospitality and support offered by Rice University during visits.
We would like to thank Mikko Salo and Matti Lassas for discussions.
We are grateful to the anonymous referees and Tuomas Hyt\"onen for valuable comments and suggestions.

\bibliographystyle{abbrv}
\bibliography{ip}

\end{document}